\documentclass[a4paper,11pt]{article}
\usepackage{amsmath}
\usepackage{amsfonts}
\usepackage{supertabular}
\usepackage[latin9]{inputenc}
\usepackage[english]{varioref}
\usepackage{dcolumn}
\usepackage[height=22cm , width = 16cm , top = 4cm , left = 3cm, a4paper]{geometry}
\usepackage[a4paper]{geometry}
\usepackage[final]{graphicx}
\usepackage{epsfig}
\usepackage{pstricks}
\usepackage{psfrag}
\usepackage{rotating}
\usepackage{supertabular}
\usepackage{booktabs}
\usepackage{delarray}
\usepackage{rotating}
\usepackage{subfigure}
\usepackage{nextpage}
\usepackage{layout}
\usepackage{amsthm}
\usepackage{dsfont}
\usepackage{hyperref}
\usepackage[hypcap]{caption}
\usepackage{color}
\numberwithin{equation}{section}

\theoremstyle{plain}
\newtheorem{thm}{Theorem}[section]

\newtheorem{lem}[thm]{Lemma}

\newtheorem{defn}[thm]{Definition}

\newcommand{\enter}{\bigskip}

\date{}
\begin{document}
 \author{{Prasanta Kumar Barik and Ankik Kumar Giri}
  }

\title {Existence and uniqueness of weak solutions to the singular kernels coagulation equation with collisional breakage }

\maketitle

\hrule \vskip 11pt

\begin{quote}
{\small {\em\bf Abstract.} In this article, we investigate the existence and uniqueness of weak solutions to the continuous coagulation equation
 with collisional breakage for a class of unbounded collision kernels and distribution function. The collision kernels and distribution functions may have a singularity on both the coordinate axes. The proof of the existence result is based on a classical weak $L^1$ compactness method applied to suitably chosen to approximate equations. The question of uniqueness is also shown for some restricted class of collision kernels.\enter
}
\end{quote}

\vskip 11pt \hrule

\section{Introduction}\label{existintroduction1}
Coagulation-fragmentation equations (CFEs) are used as models that describe the dynamics of many physical phenomena in which two or more particles can aggregate via a collision between particles to form bigger ones or  breakage into smaller pieces. The coagulation event occurs in different chemical, biological and physical processes such as colloidal aggregation, aggregation of red blood cells, crystallization and polymerization, for instance. But, the breakage process arises at least two different cases that depend on the breakage behaviour of particles. The breakage of particles may happen due to the collision between a pair of particles (collision-induced breakage) named as \emph{nonlinear breakage} or other than the interaction between a pair of particles (external forces or spontaneously) known as the \emph{linear breakage}. For example, due to the particle-wall interaction, chemical reactions or fluid shear, linear fragmentation arises where some of the examples in which the nonlinear breakage occurs are rain droplet breakage, the formation of stars and planets and many more. In this article, we have considered the continuous coagulation and nonlinear breakage processes.\\

  The continuous coagulation equation with collisional breakage \cite{Brown:1995, Barik:2018Weak, Safronov:1972, Vigil:2006, Wilkins:1982} for the concentration $g=g(\mu, t)$ of particle of volume $\mu \in \mathbb{R}_{+}:=(0, \infty)$ at time $t\ge 0$ read, as
\begin{eqnarray}\label{sccecfe}
\frac{\partial g}{\partial t}  =\mathcal{B}_c(g)-\mathcal{D}_{cb}(g) +\mathcal{B}_b(g),
\end{eqnarray}
with the following initial data
\begin{align}\label{sin1}
g(\mu, 0) = g_{in}(\mu)\geq 0\ \ \mbox{a.e.},
\end{align}
  where
\begin{eqnarray*}\label{sC1}
\mathcal{B}_c (g) (\mu, t) := \frac{1}{2}\int_0^{\mu} E(\mu-\nu, \nu) \Psi (\mu-\nu, \nu) g(\mu-\nu, t)g(\nu, t)d\nu,
\end{eqnarray*}
\begin{eqnarray*}\label{sC2}
\mathcal{D}_{cb}(g)(\mu, t) : = \int_{0}^{\infty}   \Psi(\mu, \nu) g(\mu, t) g(\nu, t)d\nu,
\end{eqnarray*}
 \begin{eqnarray*}\label{sB1}
\mathcal{B}_b (g) (\mu, t) :=  \frac{1}{2} \int_{\mu}^{\infty} \int_{0}^{\nu}P(\mu | \nu-\tau; \tau)  E'(\nu-\tau, \tau ) \Psi(\nu-\tau, \tau)g(\nu-\tau, t)g(\tau, t)d\tau d\nu,
\end{eqnarray*}

and $E(\mu, \nu)+E'(\mu, \nu)=1$ with $0\leq E, E' \le 1$.\\

Here, the collision kernel $\Psi(\mu, \nu)$ accounts for the rate at which a particle of volume $\mu$ and a particle of volume $\nu$ collide which is symmetric with respect to $\mu$ $\&$ $\nu$ and a non-negative measurable function on ${\mathbb{R}_{+}} \times {\mathbb{R}_{+}}$. Since each collision must result in either coalescence or breakup. Thus, let $E(\mu, \nu )$ denotes the probability that the two colliding particles $\mu$ and $\nu$ aggregate into a single one $\mu+\nu$ whereas $E'(\mu, \nu )$ describes the probability that the two colliding particles $\mu$ and $\nu$ breakage into two or more daughter particles with possible transfer of mass or elastic collisions between two fragments during the collision. Moreover, $E'$ is complementary with $E$. In addition, both $E$ and $E'$ are symmetric in nature. The distribution function $P(\mu | \nu; \tau)$ describes the contribution of particles of volume $\mu$ produced from the collisional breakage event arising from the interaction between a pair of particles of volumes $\nu$ and $\tau$ which is also a nonnegative symmetric function in nature with respect to second and third variables, i.e.  $P(\mu| \nu; \tau)=P(\mu | \tau; \nu) \geq 0$. Clearly, this distribution function satisfies
\begin{equation*}
P(\mu| \nu; \tau)\begin{cases}
=0,\ \ & \text{if}\ \mu > \nu + \tau, \\
\geq 0,\ \ &  \text{if}\ 0< \mu \leq \nu +\tau.
\end{cases}
\end{equation*}
Furthermore, the mass conservation property preserves during the collisional breakage event, i.e.
\begin{align}\label{smass1}
\int_{0}^{\nu+\tau} \mu P(\mu | \nu; \tau)d\mu =\nu+\tau,\ \ \text{for all}\ \ \nu>0\ \& \ \tau >0.
\end{align}
In addition, let us mention another property of the distribution function, as
the total number of particles resulting from the collisional breakage process is given by
\begin{align}\label{sTotal particles}
\int_{0}^{\nu+\tau}P(\mu | \nu; \tau)d\mu = N,\ \text{for all}\  \nu>0\ \& \ \tau>0.
\end{align}

The first term, $\mathcal{B}_c$ of (\ref{sccecfe}) describes the increase the formation of particles of volume $\mu$ by coalescence, and the last term, $\mathcal{B}_b$ of (\ref{sccecfe}) represents the birth of particles of volume $\mu$ due to collisional breakage. In addition, the factor $1/2$ appears in $\mathcal{B}_c$ and $\mathcal{B}_b$ to neglect  double counting of formation of particles of volume $\mu$ due to both the coalescence and the collisional breakage events. The second term, $\mathcal{D}_{cb}$ of (\ref{sccecfe}) shows the death of particles of volume $\mu$ due to both the coalescence and collisional breakage. On the other hand, the term $\mathcal{D}_{cb}$ can be expressed as
 \begin{eqnarray*}
\mathcal{D}_{cb}(g)(\mu, t) : = \int_{0}^{\infty}  E(\mu, \nu) \Psi(\mu, \nu) g(\mu, t) g(\nu, t)d\nu +\int_{0}^{\infty}  E'(\mu, \nu) \Psi(\mu, \nu) g(\mu, t) g(\nu, t)d\nu.
\end{eqnarray*}

Now, consider some special cases arise from the continuous coagulation and collisional breakage equations.
\begin{itemize}
  \item If $E(\mu, \nu)=1$, then equation (\ref{sccecfe}) converts into the classical continuous \emph{Smoluchowski coagulation equation}, see \cite{Barik:2018Mass}.\\
  \item  If the collision of a particle of volume $\mu$ with a particle of volume $\nu$ results in either the coalescence of both in an $(\mu+\nu)$ or in an elastic collision leaving the incoming clusters unchanged. In that case $P(\mu |\mu;\nu)=P(\nu|\mu;\nu)=1$ and $P(\mu^{\ast}|\mu;\nu)=0$ if $\mu^{\ast} \notin \{ \mu, \nu \} $. Again, (\ref{sccecfe}) also becomes the continuous version of Smoluchowski coagulation equation with coagulation kernel $(E(\mu, \nu) \Psi(\mu, \nu))$.\\
  \item  By substituting $P(\mu|\nu;\tau)=\chi_{[\mu, \infty)}(\nu)B(\mu|\nu; \tau) + \chi_{[\mu, \infty)}(\tau)B(\mu|\tau; \nu)$ into (\ref{sccecfe}), it can easily be seen that (\ref{sccecfe}) transfer into the \emph{pure nonlinear breakage equation}, see \cite{Cheng:1990, Cheng:1988, Ernst:2007, Kostoglou:2000}.

\end{itemize}

Next, it is important to mention some physical properties, i.e. moments of concentration of particles. Let $\mathcal{M}_q$ denotes the $q^{th}$
moment of the concentration of particles $g(\mu,t)$, which is defined as
\begin{align*}
\mathcal{M}_q(t)=\mathcal{M}_q(g(\mu, t)) := \int_0^{\infty} \mu^q g(\mu, t)d\mu,\ \ \text{ where}\ \ q \in (-\infty, \infty).
\end{align*}
The total number of particles and total mass of particles are denoted by $\mathcal{M}_0$ and $\mathcal{M}_1$, respectively.
In coagulation process, the total number of particles, $\mathcal{M}_0$, decreases whereas in collisional breakage process, $\mathcal{M}_0$, increases with time. However, the total mass (volume) of the system may or may not be conserved during the coagulation and collisional breakage processes that depends on the nature of the coagulation kernel ($E \Psi$) and breakup kernel ($E' \Psi$). It is worth to mention that the negative moments are also very useful in handling the case of some physical singular collision kernels such as Smoluchowski collision kernel for Brownian motion and Granulation kernel for fluidized bed etc. which have been discussed in \cite{Camejo:2013, Camejo:2015I, Camejo:2015II, Escobedo:2004, Norris:1999}.\\

\subsection{Literature Overview}
Before getting into more details of the present work, let us first recall available literature related to the coagulation equation with
collisional breakage. There is a vast literature available on the well-posedness of the continuous coagulation and linear breakage
 models. In \cite{Giri:2011, Barik:2017Anote, Escobedo:2003, Stewart:1989, Stewart:1990}, the authors have discussed the existence and uniqueness of solutions to the continuous coagulation and linear breakage equations with nonsingular coagulation kernels under different growth conditions on coagulation and breakage kernels, whereas in \cite{ Camejo:2015I, Camejo:2015II, Canizo:2011, Escobedo:2006, Fournier:2005, Norris:1999}, the existence and uniqueness results on singular coagulation kernels have been established. The existence of self-similar solutions to continuous coagulation and linear breakage models have been discussed in \cite{Canizo:2011, Escobedo:2006, Fournier:2005, Norris:1999} whereas, in \cite{Camejo:2015I} and \cite{Camejo:2015II}, respectively, the existence and uniqueness of weak solutions to Smoluchowski coagulation equations and CFEs have been shown. However, there are a few number of articles in which the collisioal breakage or nonlinear fragmentation model have been considered, see \cite{Cheng:1990, Cheng:1988, Ernst:2007, Kostoglou:2000}. In \cite{Cheng:1990, Cheng:1988, Ernst:2007, Kostoglou:2000}, authors have demonstrated scaling solutions as well as asymptotic behaviour of solutions to the \emph{pure nonlinear breakage models}. Moreover, they have also found analytical solutions for some specific collision and breakup kernels. In 1972, Safronov has proposed a new kinetic model which is a \emph{continuous coagulation and collisional breakage model}, see \cite{Safronov:1972} which has been further studied by Wilkins in \cite{Wilkins:1982}. This model becomes the continuous nonlinear fragmentation model if $w(\mu, \nu)=0$, where $w(\mu, \nu)$ is the probability that particles of volume $\mu$ and $\nu$ sticks together during the collision. In $2001$, Lauren\c{c}ot and Wrzosek have discussed the existence and uniqueness of weak solutions by using a weak $L^1$ compactness argument to the discrete coagulation and collisional breakage model. They have also studied mass conservation, gelation and large time behaviour of the solutions. This model becomes the discrete nonlinear fragmentation model for the same probability of $w(\mu, \nu)$. Recently, in \cite{Barik:2018Weak} Barik and Giri have shown the existence of weak solutions for a particular classes of nonsingular collision kernels. The main novelty of the present work is to include the singular collision kernels to show the existence of weak solutions to continuous coagulation and collisional breakage models. Unfortunately, we have not covered our previous result given in \cite{Barik:2018Weak}. Moreover, a uniqueness result is shown which is motivated by \cite{Stewart:1990}.\\

 This paper is organized as follows:
 \begin{itemize}
   \item In Section $2$, we construct some assumptions on the $\Psi $ $\&$ $g_{\text{in}}$, definition and together with the main results of this article.
   \item  The existence of weak solutions to  (\ref{sccecfe})--(\ref{sin1}) has been established by using a weak $L^1$
  compactness method in Section $3$.
   \item In the last section, a uniqueness result of weak solutions to (\ref{sccecfe})--(\ref{sin1}) is established for a class of restricted collision kernel which is motivated by \cite{Stewart:1990}.
 \end{itemize}


\section{Preliminaries and results}
We first consider the basic assumptions on the collision kernel $\Psi$, initial data $g_{\text{in}}$, probability $E$ and the distribution function $P$ throughout the paper. Assume that the collision kernel $\Psi$ satisfies
\begin{align}\label{collisionKernels}
 \Psi(\mu, \nu) \leq k \frac{ (1+\mu)^{\omega}(1+\nu)^{\omega}}{(\mu +\nu)^{\sigma}}\ \   \text{for all}\  (\mu, \nu)\in {\mathbb{R}_{+}}
\times {\mathbb{R}_{+}},
\end{align}
$ 0 \le \omega -\sigma  < 1, \sigma \in (0, {1}/{2} ), 0 \le \omega <1$ and some constant $k \geq 0$.\\
Next, we describe that the probability $E$ enjoys the following relation for small volume particles
\begin{eqnarray}\label{probabilitys}
 E(\mu, \nu) \geq \frac{\eta(r)-2}{\eta(r)-1},\ \ \forall (\mu, \nu) \in (0, 1) \times (0, 1),
\end{eqnarray}
where $ \eta(r)$ is a real valued function of $r$ such that $ \eta(r) \geq N \geq 2$, for $0  \leq r<1$ and $N$ total number of fragments is obtained after the collision between a pair of particles (given in \eqref{sTotal particles}).\\
\\
We further assume the following four assumptions on the distribution function: recalling same $\eta(r)$ and $r$ from \eqref{probabilitys}, we have
\begin{eqnarray}\label{distribution1s}
 \int_0^{\nu} \mu^{-r}  P(\mu|\nu -\tau;\tau)d\mu \leq \eta(r) \nu^{-r}.
\end{eqnarray}

 For $\nu \in (0, \lambda)$  and any small measurable subset $A$ of  $(0, \lambda)$ with $|A| \leq \delta$, there exist $\theta_1$,
 $\theta_2$ $\in (0, 2\sigma]$ such that
\begin{eqnarray}\label{distribution2s}
 \int_{0}^{\nu} \chi_{A}(\mu) P(\mu|\nu-\tau;\tau)d\mu \leq \Omega_1(|A|)\nu^{-\theta_1}, \ \ \ \text{where}\ \ \lim_{\delta \to 0} \Omega_1 ( \delta )=0,
\end{eqnarray}
and
\begin{eqnarray}\label{distribution3s}
 \int_{0}^{\nu} \chi_{A}(\mu)\mu^{-\sigma} P(\mu|\nu-\tau;\tau)d\mu \leq \Omega_2(|A|)\nu^{-\theta_2}, \ \ \text{where}\ \
 \lim_{\delta \to 0} \Omega_2 ( \delta)=0.
\end{eqnarray}
Here, $|A|$ denotes the Lebesgue measure of $A$ and $\chi_{A}$ is the characteristic function on a set $A$ given by
\begin{equation*}
\chi_{A(\mu)}:=\begin{cases}
1,\ \ & \text{if}\ \mu\in A, \\
0,\ \ &  \text{if}\ \mu \notin A.
\end{cases}
\end{equation*}
 For $\nu +\tau > \lambda$ and $\mu\in (0, \lambda)$ for some $\tau_2 \in [0,1)$ with an additional restriction $\tau_2 +\sigma <1$, we have
\begin{eqnarray}\label{distribution4s}
 P(\mu|\nu;\tau) \leq k'(\lambda)\mu^{-\tau_2},\ \ \text{ where }\ k'(\lambda)>0.
 \end{eqnarray}

In order to prove the uniqueness result, we need the further restriction on collision kernel given in \eqref{collisionKernels}, as
\begin{eqnarray}\label{collisionKerneluniques}
\Psi(\mu, \nu) \le \frac{k}{(\mu+\nu)^{\sigma}}\ \ \text{for}\ k \ge 0\ \ \text{and}\ \sigma \in (0, 1/2).
\end{eqnarray}

Finally, we turn to assume the initial data $g_{\text{in}}$ satisfies
\begin{eqnarray}\label{initialdatas}
g_{\text{in}} \in \mathcal{S}^{+},
\end{eqnarray}
where $\mathcal{S}^{+}$ the positive cone of the Banach space
\begin{eqnarray*}
\mathcal{S}:= L^1(\mathbb{R}_{+}; (1+\mu+\mu^{-2\sigma}) d\mu),        
\end{eqnarray*}
endowed with the norm
\begin{eqnarray*}
 \|g\|_{\mathcal{S}}:=\int_{0}^{\infty}(1+\mu+\mu^{-2\sigma})|g(\mu)|d\mu.
\end{eqnarray*}

It can easily be seen that $\mathcal{S}$ is a Banach with respect to norm $\|\cdot\|_{\mathcal{S}}$, see \cite{Camejo:2013}.
The existence and uniqueness of weak solutions to (\ref{sccecfe})--(\ref{sin1}) will be shown on the Banach space $\mathcal{S}^+$.
Now, we formulate weak solutions to (\ref{sccecfe})--(\ref{sin1}) by using the following definition.
\begin{defn}\label{sdef1} A solution $g$ of (\ref{sccecfe})--(\ref{sin1}) is a non-negative continuous function $g: [0,T]\to \mathcal{S}^+$ such that, for a.e. $\mu \in \mathbb{R}_{+}$ and all $t\in [0,T]$,\\
$(i)$  the following integrals are finite
\begin{align*}
\int_{0}^{t}\int_{0}^{\infty}\Psi(\mu, \nu)g(\nu, s)d\nu ds<\infty,\ \  \mbox{and} \ \ \int_{0}^{t} \mathcal{B}_b(g)(\mu, s)ds, 
\end{align*}
$(ii)$  the function $g$ satisfies the following weak formulation of (\ref{sccecfe})--(\ref{sin1})
\begin{align*}
g(\mu, t)&=g_{\text{in}}(\mu)+ \int_{0}^{t}[\mathcal{B}_c(g)(\mu, s)-\mathcal{B}_{cb}(g)(\mu, s)+\mathcal{B}_b(g)(\mu, s)]ds,
\end{align*}
where $T \in \mathbb{R}_{+}$.
\end{defn}

 Let us now verify \eqref{collisionKernels}  by taking the following example of collision kernel:\\

$(1)$ based on kinetic theory \cite{Aldous:1999},
\begin{eqnarray*}
 \Psi(\mu, \nu)=k(\mu^{1/3}+\nu^{1/3})(\mu \nu)^{1/2} (\mu +\nu )^{-3/2}.
\end{eqnarray*}

%
%

 Next, let us consider the following distribution function
  \begin{eqnarray*}
  P(\mu|\nu;\tau)= &(\theta +2)\frac{\mu^{\theta}}{(\nu+\tau)^{\theta +1}},\ \text{where}\ -2 <\theta \leq 0\  \text{and}\ \mu<\nu+\tau.
 \end{eqnarray*}
By inserting $P$ into \eqref{sTotal particles}, it is observed that the distribution function $P$ is physically meaningful only if $-1< \theta \leq 0$. For $\theta =0$, this gives the binary breakage and for $-1< \theta < 0$, we get the finite number of daughter particles which is denoted by $N$ and written as $N = \frac{\theta +2}{\theta +1}$. One can also easily see that for $\theta \in (-2, -1]$, we get an infeasible number of daughter particles. Therefore, in this article, we have considered $\theta \in (-1, 0]$. Let us verify \eqref{distribution1s}--\eqref{distribution4s} by considering the above example of distribution function $P$. First, we check \eqref{distribution1s}, as
\begin{align*}
 \int_0^{\nu} \mu^{-r}  P(\mu|\nu -\tau;\tau)d\mu = &\frac{(\theta +2)}{\nu^{\theta +1}} \int_0^{\nu} \mu^{\theta-r}  d\mu \nonumber\\
 = &\frac{(\theta +2)}{\nu^{\theta +1}}  \frac{\mu^{\theta-r+1}}{\theta-r+1}\bigg|_{0}^{\nu}, \ \ \text{provided that}\  \theta-r+1>0\nonumber\\
  = &\frac{(\theta +2)}{\nu^{\theta +1}}  \frac{\nu^{\theta-r+1}}{\theta-r+1} \le \frac{(\theta +2)}{\theta-r +1}  \nu^{-r}=:\eta(r)\nu^{-\theta_1}.
\end{align*}

Next, we verify \eqref{distribution2s}, as follows: suppose $q>1$, then applying H\"{o}lder's inequality, we simplify
 \begin{align*}
 \int_0^{\nu} \chi_A(\mu) P(\mu|\nu-\tau;\tau)d\mu 
 \leq &\frac{(\theta +2)}{\nu^{\theta +1}} |A|^{\frac{q-1}{q}} \bigg[\int_0^{\nu}  \mu^{q\theta}d\mu \bigg]^{\frac {1} {q}}\nonumber\\
  = &\frac{(\theta +2)}  {(q\theta+1)^{ {1/q}} }  \frac{1 }{\nu^{\theta +1}} |A|^{\frac{q-1}{q}} \nu^{\theta+1/q},\ \ \text{where}\ \ \theta> -1/ q \nonumber\\
   = &\frac{(\theta +2)}  {(q\theta+1)^{ {1/q}} }  |A|^{\frac{q-1}{q}} \nu^{-\theta_1}=\Omega_1(|A|)\nu^{-\theta_1},
 \end{align*}
 where $\theta_1:=1-\frac{1}{q} \in (0, 2\sigma]$.\\

 Similarly, by applying H\"{o}lder's inequality for $p>1$, \eqref{distribution3s} can be verified as
 \begin{align*}
 \int_0^{\nu} \chi_A(\mu){\mu}^{-\sigma} P(\mu|\nu-\tau;\tau)d\mu
 \leq &(\theta +2)\frac{1}{\nu^{\theta +1}} |A|^{\frac{p-1}{p}} \bigg[\int_0^{\nu}  \mu^{p(\theta-\sigma)}d\mu \bigg]^{ \frac{1}{p}}\nonumber\\
  = &\frac{(\theta +2)}  {(p(\theta-\sigma)+1)^{ {1/p}} }  \frac{1 }{\nu^{\theta +1}} |A|^{\frac{p-1}{p}} \nu^{\theta-\sigma+1/p}\ \ \text{where}\ p(\theta-\sigma)+1>0\nonumber\\
   = &\frac{(\theta +2)}  {(p(\theta-\sigma)+1)^{ {1/p}} }  |A|^{\frac{p-1}{p}} \nu^{-\theta_2}=\Omega_2(|A|)\nu^{-\theta_2},
 \end{align*}
 where $\theta_2 :=1+\sigma-\frac{1}{p} \in (0, 2\sigma]$.\\

 In order to check \eqref{distribution4s}, we have
\begin{align*}
P(\mu|\nu;\tau)= (\theta +2)\frac{\mu^{\theta}}{(\nu+\tau)^{\theta +1}} \leq (\theta +2) \frac{\mu^{\theta}}{\lambda ^{1+\theta}}\leq  k'(\lambda) \mu^{-\tau_2}\ \text{for}\ \lambda < \nu +\tau,
\end{align*}
where, $\tau_2 \in [0,1)$  and  $k'(\lambda)\geq \frac{\theta +2}{\lambda^{1+\theta}}$.

Now, It is the right place to state the following existence and uniqueness results:
\begin{thm}\label{existmain theorem1}
Assume that \eqref{collisionKernels}--\eqref{distribution4s} hold. Then, for $g_{\text{in}}\in \mathcal{S}^+$, there exists a weak solution $g$ to (\ref{sccecfe})--(\ref{sin1}) on $[0, \infty)$.
\end{thm}
\begin{thm}\label{uniquemain theorem1}
Assume that the collision kernel $\Psi$ satisfies \eqref{collisionKerneluniques} $\&$ \eqref{probabilitys}--\eqref{distribution4s} hold. Suppose $g_{\text{in}}\in \mathcal{S}^+$. Then, (\ref{sccecfe})--(\ref{sin1}) has a unique weak solution on $[0, \infty)$.
\end{thm}

\section{Existence of weak solutions}\label{existexistence1}
In this section, first, we construct a sequence of functions $(\Psi_n)_{n=1}^{\infty} $ with compact support for each $n \geq 1$, such that
\begin{equation}\label{coalkernel}
\Psi_n(\mu, \nu):=\begin{cases}
\Psi(\mu, \nu),\ \ & \text{if}\  \  \mu+\nu < n\ \& \ \mu, \nu \geq 1/n, \\
\text{0},\ \ &  \ \ \ \text{otherwise}.
\end{cases}
\end{equation}
Next, we may follow as in \cite{Camejo:2015II} or \cite{Walker:2002} to show the truncated equation
\begin{align}\label{5trun cccecf}
\frac{\partial g^n }{\partial t}  = \mathcal{B}_{c}^n(g^n)-\mathcal{D}_{cb}^n (g^n)+  \mathcal{B}_{b}^n(g^n)
\end{align}
with initial data
\begin{equation}\label{5trunc ccecfin1}
g^{n}_0(\mu):=\begin{cases}
g_{\text{in}}(\mu),\ \ & \text{if}\ 1/n < \mu< n, \\
\text{0},\ \ &  \ \ \ \text{otherwise},
\end{cases}
\end{equation}
where
\begin{align*}
 & \mathcal{B}_{c}^n(g^n)(\mu, t):=\frac{1}{2} \int_{0}^{\mu}  E(\mu-\nu, \nu) \Psi_n(\mu-\nu, \nu) g^n(\mu-\nu, t)g^n(\nu, t)d\nu,\\
 &\mathcal{D}_{cb}^n(g^n)(\mu, t):= \int_{0}^{n-\mu} \Psi_n(\mu, \nu)g^n(\mu, t)g^n(\nu, t)d\nu, \nonumber\\
 & \mathcal{B}_{b}^n(g^n)(\mu, t):=\frac{1}{2} \int_{\mu}^{n}\int_{0}^{\nu} P(\mu |\nu-\tau; \tau) E'(\nu-\tau, \tau) \Psi_n(\nu-\tau, \tau) g^n(\nu-\tau, t)g^n(\tau, t)d\tau d\nu,
\end{align*}
has a unique non-negative solution $\hat{g}^n$. These family of solutions $(\hat{g}^n)_{n=1}^{\infty}$ lie in $ \mathcal{C}'([0, T];L^1(0,n))$. Additionally, it enjoys the mass conserving property for all $t\in[0, T]$, i.e.
\begin{eqnarray}\label{strunc mass1}
\int_{0}^{n}\mu \hat{g}^n(\mu, t)d\mu =\int_{0}^{n}\mu \hat{g}^n_{\text{in}}(\mu)d\mu.
\end{eqnarray}
 Furthermore, we extend the domain of truncated unique solution $\hat{g}^n$ by zero in $\mathbb{R}_{+} \times [0, T]$, as
\begin{equation}\label{5trunc soln}
\hat{g}^{n}(\mu, t):=\begin{cases}
g^n(\mu, t),\ \ & \text{if}\ 1/n < \mu < n, \\
\text{0},\ \ &  \text{if}\ \mu \geq n,
\end{cases}
\end{equation}
for $n\ge 1$ and $n \in \mathbb{N}$.\\

In the coming section, we wish to apply a classical weak compactness technique for the family of solutions $(g^n)_{n\in \mathds{N}}$ to obtain required weak solutions.

\subsection{ Weak compactness}\label{subs:wk}
Here we first show the equi-boundedness of the family $(g^n)_{n\in \mathds{N}} \subset \mathcal{S}^+ $ by proving the next lemma.
\begin{lem}\label{LemmaEquibounded}
Given any $T>0$. Then the following holds
\begin{eqnarray*}
 \int_{0}^{\infty}(1+\mu+\mu^{-2\sigma})g^n(\mu, t)d\mu \leq \mathcal{P}(T)\ \ \text{for}\   n\geq 1\  \text{and all} \ \ t\in [0,T],
\end{eqnarray*}
where $\mathcal{P}(T)$ is a positive constant depending on $T$.
\end{lem}

\begin{proof}
 Let $t\in [0,T]$ and $n\geq 1$. For $n=1$, the proof of the lemma is trivial. Now, for $n\geq 2$, using \eqref{5trunc soln} and \eqref{strunc mass1}, consider the following integral as
 \begin{align}\label{sEquibound1}
\int_0^{\infty} (1+\mu +\mu ^{-2\sigma}) g^n(\mu, t)d\mu = &\int_{1/n}^1 (1+\mu^{-2\sigma}) g^n(\mu, t)d\mu \nonumber\\
&+\int_1^n (1+\mu^{-2\sigma}) g^n(\mu, t)d\mu +\int_{1/n}^n \mu g^n(\mu,t)d\mu \nonumber\\
\leq &2\int_0^1 \mu^{-2\sigma} g^n(\mu, t)d\mu +3\|g_{\text{in}}\|_{\mathcal{S}}.
\end{align}
Thanks to \eqref{5trunc ccecfin1} and $g_{\text{in}} \in  \mathcal{S}^+$. Next, multiplying with $\mu^{-2\sigma}$ in (\ref{5trun cccecf}), taking integration from $0$ to $1$ with respect to $\mu$ and applying Leibniz's rule, we get
\begin{align}\label{sEquibound2}
\frac{d}{dt}\int_0^1 \mu^{-2\sigma} g^n(\mu, t)d\mu =\int_0^1 \mu^{-2\sigma}  [\mathcal{B}_c^n(g^n)(\mu, t)-\mathcal{D}_{cb}^n(g^n)(\mu, t)+\mathcal{B}_b^n(g^n)(\mu, t)] d\mu.
\end{align}
First, applying integral to each term of (\ref{sEquibound2}), then using Fubini's theorem, the transformation $\mu-\nu={\mu}'$ $\&$ $\nu={\nu}'$, interchanging $\mu$ and $\nu$ and symmetry of $\Psi_n$ $\&$ $E$, the first term on the right-hand side of (\ref{sEquibound2}) can be written as
\begin{align}\label{sEquibound3}
\int_0^1 \mu^{-2\sigma}  \mathcal{B}_c^n(g^n)(\mu, t)d\mu 
 =&\frac{1}{2} \int_0^1\int_{0}^{1-\mu}  (\mu+\nu)^{-2\sigma} E(\mu, \nu)  \Psi_n(\mu, \nu)g^n(\mu, t)g^n(\nu, t)d\nu d\mu.
\end{align}
Again, using Fubini's theorem, the third term on the right-hand side of (\ref{sEquibound2}) can be simplified as
\begin{align}\label{sEquibound4}
\int_0^1 \mu^{-2\sigma} & \mathcal{B}_b^n(g^n)(\mu, t)d\mu \nonumber\\ 
 =&\frac{1}{2}  \int_0^1 \int_0^{\nu} \int_0^{\nu} \mu^{-2\sigma} P(\mu| \nu-\tau; \tau)  E'(\nu-\tau, \tau) \Psi_n(\nu-\tau, \tau) g^n(\nu-\tau, t)g^n(\tau, t)d\tau d\mu d\nu  \nonumber\\
&+\frac{1}{2}  \int_1^n \int_0^{1} \int_0^{\nu} \mu^{-2\sigma} P(\mu| \nu-\tau; \tau) E'(\nu-\tau, \tau) \Psi_n(\nu-\tau, \tau) g^n(\nu-\tau, t)g^n(\tau, t)d\tau d\mu d\nu  \nonumber\\
=&:J_1^n+J_2^n.
\end{align}

Let us consider the integral $J_1^n$ by applying the repeated applications of Fubini's theorem, \eqref{distribution1s}, the transformation $\nu-\tau=\nu'$ $\&$ $\tau=\tau'$ and finally replacing $\nu\rightarrow \mu$ $\&$ $\tau\rightarrow \nu$, we obtain
\begin{align*}
J_1^n 
 \le & \frac{\eta(2\sigma)}{2}   \int_0^1 \int_0^{1-\mu} (\mu+\nu)^{-2\sigma} E'(\mu, \nu) \Psi_n(\mu, \nu) g^n(\mu, t)g^n(\nu, t) d\nu d\mu.
\end{align*}

Next, we simplify $J_2^n$ by applying \eqref{distribution1s}, replacing $\tau$, $\nu$ by $\nu$, $\mu$ respectively, Fubini's theorem, and the transformation $\mu-\nu=\mu'$ $\&$ $\nu=\nu'$, to obtain
\begin{align*}
J_2^n
\leq & \frac{1}{2}  \int_1^n  \int_0^{\nu} \int_0^{\nu}  \mu^{-2\sigma} P(\mu| \nu-\tau; \tau) E'(\nu-\tau, \tau) \Psi_n(\nu-\tau, \tau) g^n(\nu-\tau, t)g^n(\tau, t)d\mu d\tau  d\nu \nonumber\\
\leq & \frac{\eta(2\sigma)}{2}  \int_0^1  \int_{1-\mu}^{n-\mu} (\mu+\nu)^{-2\sigma} E'(\mu, \nu) \Psi_n(\mu, \nu) g^n(\mu, t)g^n(\nu, t) d\nu d\mu \nonumber\\
& + \frac{\eta(2\sigma)}{2}  \int_1^n  \int_{0}^{n-\mu} (\mu+\nu)^{-2\sigma} E'(\mu, \nu) \Psi_n(\mu, \nu) g^n(\mu, t)g^n(\nu, t) d\nu d\mu.
\end{align*}

Substituting estimates on $J_1^n $ and $J_2^n $ into (\ref{sEquibound4}), we have
\begin{align}\label{sEquibound45}
\int_0^1 \mu^{-2\sigma}  \mathcal{B}_b^n(g^n)&(\mu, t)d\mu 
\leq  \frac{\eta(2\sigma)}{2} \int_0^n \int_0^{n-\mu} (\mu+\nu)^{-2\sigma}  E'(\mu, \nu) \Psi_n(\mu, \nu) g^n(\mu, t)g^n(\nu, t) d\nu d\mu.
\end{align}

Inserting (\ref{sEquibound3}) and (\ref{sEquibound45}) into (\ref{sEquibound2}), and using Fubini's theorem, we arrange (\ref{sEquibound2}), as

\begin{align}\label{sEquibound6}
\frac{d}{dt}  \int_0^1  \mu^{-2\sigma} & g^n(\mu, t)d\mu \nonumber\\
\leq & - \int_0^1\int_{0}^{1-\mu} \bigg[1-\frac{1}{2}  E(\mu, \nu) -\frac{\eta(2\sigma)}{2} E'(\mu, \nu) \bigg] \mu^{-2\sigma} \Psi_n(\mu, \nu)g^n(\mu, t) g^n(\nu, t)d\nu d\mu \nonumber\\
&- \int_0^1\int_{1-\mu}^{1} \bigg[1-\frac{\eta(2\sigma)}{2} E'(\mu, \nu) \bigg] \mu^{-2\sigma} \Psi_n(\mu, \nu)g^n(\mu, t)g^n(\nu, t)d\nu d\mu \nonumber\\
&- \int_0^1\int_{1}^{n-\mu} \mu^{-2\sigma} \Psi_n(\mu, \nu)g^n(\mu, t)g^n(\nu, t)d\nu d\mu \nonumber\\
&+ \frac{\eta(2\sigma)}{2}  \int_0^1  \int_{1}^{n} \mu^{-2\sigma} E'(\mu, \nu) \Psi_n(\mu, \nu)g^n(\mu, t)g^n(\nu, t)d\nu d\mu \nonumber\\
& + \frac{\eta(2\sigma)}{2}  \int_1^n  \int_{0}^{n} \mu^{-2\sigma}  E'(\mu, \nu)  \Psi_n(\mu, \nu)g^n(\mu, t)g^n(\nu, t)d\nu d\mu.
\end{align}
Applying \eqref{probabilitys} to (\ref{sEquibound6}). Using the negativity of the first, second  $\&$ the third terms on the right-hand side to (\ref{sEquibound6}) and then applying \eqref{collisionKernels}, (\ref{strunc mass1}) and \eqref{5trunc ccecfin1}, we obtain
\begin{align}\label{sEquibound7}
\frac{d}{dt} \int_0^1  \mu^{-2\sigma} g^n(\mu, t)& d\mu \leq k \frac{\eta(2\sigma)}{2}  \int_0^1  \int_{1}^{n} \mu^{-2\sigma}  \frac{ (1+\mu)^{\omega} (1+\nu)^{\omega}}{(\mu +\nu )^{\sigma}}  g^n(\mu, t)g^n(\nu, t)d\nu d\mu \nonumber\\
& + k \frac{\eta(2\sigma)}{2}  \int_1^n  \int_{0}^{1} \mu^{-2\sigma}   \frac{ (1+\mu)^{\omega} (1+\nu)^{\omega}}{(\mu +\nu )^{\sigma}}g^n(\mu, t)g^n(\nu, t)d\nu d\mu \nonumber\\
& +k \frac{\eta(2\sigma)}{2}  \int_1^n  \int_{1}^{n} \mu^{-2\sigma}   \frac{ (1+\mu)^{\omega} (1+\nu)^{\omega}}{(\mu +\nu )^{\sigma}}g^n(\mu, t)g^n(\nu, t)d\nu d\mu \nonumber\\
\leq & k \eta(2\sigma)   2^{2\omega} \int_0^1  \int_{1}^{n}  \mu^{-2\sigma} \nu  g^n(\mu, t)g^n(\nu, t)d\nu d\mu \nonumber\\
& + k\frac{\eta(2\sigma)}{2}  2^{2\omega} \int_1^n  \int_{1}^{n} \mu \nu g^n(\mu, t)g^n(\nu, t)d\nu d\mu \nonumber\\
\leq & k \eta(2\sigma)   2^{2\omega} \|g_{\text{in}}\|_{\mathcal{S}} \bigg[ \int_0^{1} \mu^{-2\sigma}  g^n(\mu, t) d\mu +\frac{1}{2} \|g_{\text{in}}\|_{\mathcal{S}} \bigg].
\end{align}
Again thanks to $ g_{\text{in}} \in   \mathcal{S}^+$. Then, an application of Gronwall's inequality to (\ref{sEquibound7}) gives
\begin{eqnarray}\label{sEquibound8}
\int_0^1 \mu^{-2\sigma} g^n(\mu, t)d\mu \leq \mathcal{P}_1(T),
\end{eqnarray}
where
\begin{eqnarray*}
\mathcal{P}_1(T):=e^{aT} \|g_{\text{in}}\|_{\mathcal{S}} +
\frac{b }{a} ( e^{aT } -1 ),
\end{eqnarray*}
 $a:=  k \eta(2\sigma)   2^{2\omega} \|g_{\text{in}}\|_{\mathcal{S}} $ and $b:= k \frac{\eta(2\sigma)}{2} 2^{2\omega} \|g_{\text{in}}\|_{\mathcal{S}}^2$.
Inserting (\ref{sEquibound8}) into (\ref{sEquibound1}), we thus have
 \begin{eqnarray*}
\int_0^{\infty} (1+\mu+\mu^{-2\sigma}) g^n(\mu, t)d\mu \leq 2\mathcal{P}_1(T)+3 \|g_{\text{in}}\|_{\mathcal{S}} := \mathcal{P}(T).
\end{eqnarray*}
This completes the proof of Lemma \ref{LemmaEquibounded}.
\end{proof}


Next, the equi-integrability of the family of functions $(g^n)_{n\in \mathds{N}} \subset  \mathcal{S}^+ $ is shown in the following
lemma to apply Dunford-Pettis theorem.
\begin{lem}\label{compactlemma}
Given any $T>0$. Then, followings are hold true:\\

$(i)$ for all $t\in[0,T]$ and for any given $\epsilon> 0$, there exists a positive constant $\lambda_\epsilon >0$ (depending on $\epsilon $) such that
\begin{eqnarray*}
\sup_{n\geq 1} \left\{ \int_{\lambda_\epsilon}^{\infty}(1+\mu^{-\sigma})g^n(\mu, t)d\mu \right\} \leq \epsilon,
\end{eqnarray*}
$(ii)$ for a given $\epsilon > 0$, there exists  $\delta_\epsilon>0$ (depending on $\epsilon$) such that, for every small measurable set
$A \subset \mathbb{R}_{+}$ with $|A|\leq \delta_\epsilon$, $n \ge 1$ and $t\in [0,T]$,
\begin{eqnarray*}
\int_{A}(1+{\mu}^{-\sigma})g^n(\mu, t)d\mu < \epsilon.
\end{eqnarray*}
\end{lem}

\begin{proof}
$(i)$ For $\epsilon> 0$, let $ \epsilon:= \frac {(\lambda_{\epsilon}^{\sigma}+1)\int_{\lambda_\epsilon}^{\infty}\mu g_{\text{in}}(\mu)d\mu } { \lambda_{\epsilon}^{1+\sigma} }$.
 Then, by (\ref{strunc mass1}), for each $n\ge 1$, $g_{\text{in}} \in  \mathcal{S}^+ $ and for all $t\in[0,T]$, we have
\begin{eqnarray*}
\int_{\lambda_\epsilon}^{\infty}(1+\mu^{-\sigma})g^n(\mu, t)d\mu
 \leq \bigg[  \frac{1}{\lambda_\epsilon}+\bigg(\frac{1}{\lambda_\epsilon}\bigg)^{\sigma+1} \bigg]
\int_{\lambda_\epsilon}^{\infty}\mu g_{\text{in}}(\mu)d\mu  = \epsilon.
\end{eqnarray*}
Therefore, we obtain
\begin{eqnarray*}
\sup_{n\ge 1} \left\{ \int_{\lambda_\epsilon}^{\infty} (1+\mu^{-\sigma}) g^n(\mu, t) d\mu \right\}\leq \epsilon.
\end{eqnarray*}
This proves the first part of Lemma \ref{compactlemma}.\\

$(ii)$ Let $ \epsilon > 0$ be given. For $A \subset (0, \lambda)$, we can choose $\lambda$ such that $ \lambda \in [1, n)$ for all
$n\in \mathds{N}$ and $t \in [0,T]$, and using Lemma \ref{compactlemma} $(i)$, we thus have
\begin{eqnarray}\label{comp3 2}
\int_{\lambda}^{\infty}(1+\mu^{-\sigma})g^n(\mu, t)d\mu <  \frac {\epsilon} {2}.
\end{eqnarray}
 Fix $\lambda \ge 1$, for $n\ge 1$, $\delta \in (0, 1)$ and $t\in [0,T]$, we define
\begin{eqnarray*}
 {\rho}^n(\delta,t):=\sup \left\{\int_{0}^{\lambda} \chi_{A}(\mu)(1+\mu^{-\sigma})g^n(\mu, t)d\mu\ :\ A\subset (0,\lambda) \ \ \text{and} \ \
|A|\leq \delta \right\}.
\end{eqnarray*}
For $n\ge 1$ and $t\in [0,T]$, we estimate the following term by using (\ref{5trun cccecf}), Leibniz's rule, the non-negativity of $g^n$, Fubini's theorem and the transformation $\mu-\nu=\mu'$ $\&$ $\nu=\nu'$, as
\begin{align}\label{scomp3 1}
&\frac{d}{dt}  \int_{0}^{\lambda}  \chi_{A}(\mu)  (1+\mu^{-\sigma}) g^n(\mu, t)d\mu\nonumber\\
\leq &\frac{1}{2}\int_{0}^{\lambda}\int_{0}^{\lambda-\mu}\chi_{A}(\mu+\nu)(1+(\mu+\nu)^{-\sigma})  \Psi_n(\mu, \nu)g^n(\mu, t)g^n(\nu, t)d\nu d\mu \nonumber\\
&+\frac{1}{2} \int_{0}^{\lambda}\int_{0}^{\nu}\int_{0}^{\nu}\chi_{A}(\mu)(1+\mu^{-\sigma})P(\mu|\nu-\tau;\tau)\Psi_n(\nu-\tau, \tau)g^n(\nu-\tau, t)g^n(\tau,t)d\mu d\tau d\nu \nonumber\\
&+\frac{1}{2} \int_{\lambda}^{n}\int_{0}^{\nu}\int_{0}^{\lambda} \chi_{A}(\mu)(1+\mu^{-\sigma}) P(\mu|\nu-\tau;\tau)\Psi_n(\nu-\tau, \tau)g^n(\nu-\tau, t)g^n(\tau,t)d\mu d\tau d\nu.
 \end{align}
 We denote the first, the second and the third terms on the right-hand side of (\ref{scomp3 1}) by $J_3^n(t)$, $J_4^n(t)$ and $J_5^n(t)$, respectively.
 Next, we estimate each $J_i^n(t)$, for $i=3,4,5$ separately. Now, $J_3^n(t)$ can be estimated, by using \eqref{collisionKernels} and Lemma
\ref{LemmaEquibounded} in a similar way as given in Camejo and Warnecke \cite{Camejo:2015II} with little modifications, as
\begin{eqnarray*}\label{J3}
J_3^n(t) \leq k \mathcal{P}(T)G(\lambda){\rho}^n (\delta , t),
\end{eqnarray*}
where $G(\lambda):=\frac{1}{2}(1+\lambda^{\sigma})(1+\lambda)^{2\omega}$. The term $J_4^n(t)$ is evaluated, by using \eqref{distribution2s} and \eqref{distribution3s}, as
\begin{align}\label{J4}
J_{4}^n(t) 
\leq & \frac{1}{2}\Omega_1(|A|) \int_{0}^{\lambda}\int_{0}^{\nu} \nu^{-\theta_1}\Psi_n(\nu-\tau, \tau)g^n(\nu-\tau, t)g^n(\tau,t)d\tau  d\nu \nonumber\\
& +\frac{1}{2} \Omega_2(|A|) \int_{0}^{\lambda}\int_{0}^{\nu}  \nu^{-\theta_2} \Psi_n(\nu-\tau, \tau)g^n(\nu-\tau, t)g^n(\tau,t)d\tau  d\nu.
\end{align}
Applying Fubini's theorem to \eqref{J4}, and using the transformation $\nu-\tau=\nu'$ $\&$ $\tau=\tau'$, \eqref{collisionKernels} and Lemma \ref{LemmaEquibounded} and , we have
\begin{align*}
J_{4}^n(t) 
\leq & \frac{1}{2}\Omega_1(|A|)k \int_{0}^{\lambda}\int_{0}^{\lambda } ( \nu+\tau)^{-\theta_1} \frac{(1+\nu)^{\omega} (1+ \tau)^{\omega}}{(\nu+\tau)^{\sigma}} g^n(\nu, t)g^n(\tau, t)d\nu d\tau \nonumber\\
& +\frac{1}{2} \Omega_2(|A|)k \int_{0}^{\lambda}\int_{0}^{\lambda}  (\nu+\tau)^{-\theta_2} \frac{(1+\nu)^{\omega} (1+ \tau)^{\omega}}{(\nu+\tau)^{\sigma}}g^n(\nu, t)g^n(\tau,t)d\nu  d\tau \nonumber\\
\leq & \frac{1}{2}\Omega_1(|A|)k (1+\lambda)^{2\omega} \int_{0}^{\lambda}\int_{0}^{\lambda }  \nu^{-\theta_1} \tau^{-\sigma} g^n(\nu, t)g^n(\tau, t)d\nu  d\tau \nonumber\\
& +\frac{1}{2} \Omega_2(|A|)k (1+\lambda)^{2\omega} \int_{0}^{\lambda}\int_{0}^{\lambda}  \nu^{-\theta_2} \tau^{-\sigma}g^n(\nu, t)g^n(\tau,t)d\nu d\tau \nonumber\\
\leq & \frac{1}{2} k (1+\lambda)^{2\omega} [\Omega_1(|A|)+\Omega_2(|A|) ] \mathcal{P}(T)^2.
\end{align*}

 Similarly, $J_5^n(t)$ can be estimated, by using repeated applications of Fubini's theorem,  \eqref{distribution4s},  $\nu-\tau=\nu'$ $\&$ $\tau=\tau'$, \eqref{collisionKernels} and Lemma \ref{LemmaEquibounded}, as
\begin{align}\label{J5}
J_{5}^n(t) 
\leq &\frac{1}{2} k'(\lambda) \int_{\lambda}^{n} \int_{0}^{\nu} \int_{0}^{\lambda}\chi_{A}(\mu) \mu^{-\tau_2} \Psi_n(\nu-\tau, \tau)g^n(\nu-\tau, t)g^n(\tau,t)d\mu d\tau d\nu \nonumber\\
&+\frac{1}{2} k'(\lambda) \int_{\lambda}^{n} \int_{0}^{\nu} \int_{0}^{\lambda} \chi_{A}(\mu) \mu^{-\sigma -\tau_2} \Psi_n(\nu-\tau, \tau)g^n(\nu-\tau, t)g^n(\tau,t)d\mu d\tau d\nu \nonumber\\
\leq &\frac{1}{2} k'(\lambda) k \int_{0}^{\lambda} \int_{\lambda-\tau}^{n-\tau} \int_{0}^{\lambda}\chi_{A}(\mu) \mu^{-\tau_2}  \frac{(1+\nu)^{\omega} (1+ \tau)^{\omega}}{(\nu+\tau)^{\sigma}} g^n(\nu, t)g^n(\tau,t) d\mu d\nu d\tau \nonumber\\
 &+\frac{1}{2} k'(\lambda) k \int_{\lambda}^{n} \int_{0}^{n-\tau} \int_{0}^{\lambda}\chi_{A}(\mu) \mu^{-\tau_2}  \frac{(1+\nu)^{\omega} (1+ \tau)^{\omega}}{(\nu+\tau)^{\sigma}} g^n(\nu, t)g^n(\tau,t) d\mu d\nu d\tau \nonumber\\
&+\frac{1}{2} k'(\lambda) k \int_{0}^{\lambda} \int_{\lambda-\tau}^{n-\tau} \int_{0}^{\lambda}\chi_{A}(\mu) \mu^{-\sigma-\tau_2} \frac{(1+\nu)^{\omega} (1+ \tau)^{\omega}}{(\nu+\tau)^{\sigma}} g^n(\nu, t)g^n(\tau,t) d\mu d\nu d\tau \nonumber\\
 &+\frac{1}{2} k'(\lambda) k \int_{\lambda}^{n} \int_{0}^{n-\tau} \int_{0}^{\lambda}\chi_{A}(\mu) \mu^{-\sigma-\tau_2}  \frac{(1+\nu)^{\omega} (1+ \tau)^{\omega}}{(\nu+\tau)^{\sigma}} g^n(\nu, t)g^n(\tau,t) d\mu d\nu d\tau\nonumber\\
\leq & \frac{1}{2}k'(\lambda) k \mathcal{P}(T)^2 [(1+\lambda)^{\omega}+ 1/{\lambda}^{\sigma}] \bigg[  \int_{0}^{\lambda}\chi_{A}(\mu) \mu^{-\tau_2}d\mu +\int_{0}^{\lambda}\chi_{A}(\mu) \mu^{-\sigma -\tau_2}d\mu \bigg].
 \end{align}

Now, for $ \frac{1+\sigma+\tau_2}{1-\sigma-\tau_2}>1$ and choose $u>1$ such that $ u\tau_2<1$, then applying H\"older's inequality to \eqref{J5}, we get
\begin{align*}
J_{5}^n(t) \leq &  \frac{1}{2}k'(\lambda) k \mathcal{P}(T)^2 [(1+\lambda)^{\omega}+ 1/{\lambda}^{\sigma}] \bigg\{  |A|^{\frac{u-1}{u}} \left( \int_0^{\lambda} \mu^{-u\tau_2} d\mu \right)^{\frac{1}{u}}\\
 &+  |A|^{\frac{1-\sigma -\tau_2}{1+\sigma +\tau_2}} \bigg( \int_{0}^{\lambda}  \mu^{-(1+\sigma +\tau_2)/2} d\mu \bigg)^{\frac{2(\sigma+\tau_2)}{1+\sigma +\tau_2}} \bigg\}\\
 \leq &   \frac{1}{2}k'(\lambda) k \mathcal{P}(T)^2 [(1+\lambda)^{\omega}+ 1/{\lambda}^{\sigma}] \bigg\{  {\delta}^{\frac{u-1}{u}}  \left(  \frac{\lambda^{1-u\tau_2}}   {{1-u\tau_2}} \right)^{\frac{1}{u}}\nonumber\\
  &~~~~~~~~~~~~~~~~~~~~~~~~~~+  {\delta}^{\frac{1-\sigma -\tau_2}{1+\sigma +\tau_2}}   \bigg( \frac{ \lambda^{(1-(\sigma +\tau_2))/2}} { (1-\sigma-\tau_2)/2 } \bigg) ^{ \frac{2(\sigma+\tau_2)}{(1+\sigma+\tau_2)}      } \bigg\}.
\end{align*}
Inserting the estimates on $J_{3}^n(t)$, $J_{4}^n(t)$ and $J_{5}^n(t)$ into (\ref{scomp3 1}), we obtain
\begin{align}\label{comp3 3}
 \frac{d}{dt}\int_{0}^{\lambda} \chi_{A}(\mu)  (1+\mu^{-\sigma})  g^n(\mu, t)d\mu \leq &k \mathcal{P}(T)G(\lambda)\rho^n (\delta , t)+ \frac{1}{2} k (1+\lambda)^{2\omega} [\Omega_1(|A|)+\Omega_2(|A|) ] \mathcal{P}(T)^2\nonumber\\
&+ \frac{1}{2}k'(\lambda) k \mathcal{P}(T)^2 [(1+\lambda)^{\omega}+ 1/{\lambda}^{\sigma}]  \bigg\{  {\delta}^{\frac{u-1}{u}}  \left(  \frac{\lambda^{1-u\tau_2}}   {{1-u\tau_2}} \right)^{\frac{1}{u}}\nonumber\\
  &+  {\delta}^{\frac{1-\sigma -\tau_2}{1+\sigma +\tau_2}}   \bigg( \frac{ \lambda^{(1-(\sigma +\tau_2))/2}} { (1-\sigma-\tau_2)/2 } \bigg) ^{ \frac{2(\sigma+\tau_2)}{(1+\sigma+\tau_2)}      } \bigg\}.
\end{align}

Next, integrating (\ref{comp3 3}) with respect to time from $0$ to $t$ and taking supremum over all $A$ such that $A \subset (0, \lambda)$ with $|A|$ $\leq \delta$ and using $ g_{\text{in}} \in \mathcal{S}^+$, we estimate
\begin{align*}
\rho^n(\delta, t) \leq & \rho^n(\delta, 0)+k \mathcal{P}(T)G(\lambda)\int_0^t \rho^n (\delta , s)ds + \frac{1}{2} k (1+\lambda)^{2\omega} [\Omega_1(|A|)+\Omega_2(|A|) ] \mathcal{P}(T)^2T\\
&+ \frac{1}{2}k'(\lambda) k \mathcal{P}(T)^2 T [(1+\lambda)^{\omega}+ 1/{\lambda}^{\sigma}]  \bigg\{  {\delta}^{\frac{u-1}{u}}  \left(  \frac{\lambda^{1-u\tau_2}}   {{1-u\tau_2}} \right)^{\frac{1}{u}}\nonumber\\
  &~~~~~~~~~~~~~~~~~~~~~~~~~~+  {\delta}^{\frac{1-\sigma -\tau_2}{1+\sigma +\tau_2}}   \bigg( \frac{ \lambda^{(1-(\sigma +\tau_2))/2}} { (1-\sigma-\tau_2)/2 } \bigg) ^{ \frac{2(\sigma+\tau_2)}{(1+\sigma+\tau_2)}      } \bigg\},\  \ t\in [0,T].
\end{align*}
Now, applying Gronwall's inequality, we obtain
\begin{eqnarray}\label{rhon}
\rho^n(\delta,t) \leq C_1 (\delta, \lambda) \exp(k \mathcal{P}(T)G(\lambda)T),\ \ \ t \in [0,T],
\end{eqnarray}
where
\begin{align*}
C_1 (\delta, \lambda):=&\rho^n(\delta, 0)+\frac{1}{2} k (1+\lambda)^{2\omega} [\Omega_1(|A|)+\Omega_2(|A|) ] \mathcal{P}(T)^2T
+  \frac{1}{2}k'(\lambda) k \mathcal{P}(T)^2T\\ & \times [(1+\lambda)^{\omega}+ 1/{\lambda}^{\sigma}] \bigg\{  {\delta}^{\frac{u-1}{u}}  \left(  \frac{\lambda^{1-u\tau_2}}   {{1-u\tau_2}} \right)^{\frac{1}{u}}\nonumber\\
  &~~~~~~~~~~~~~~~~~~~~~~~~~~+  {\delta}^{\frac{1-\sigma -\tau_2}{1+\sigma +\tau_2}}   \bigg( \frac{ \lambda^{(1-(\sigma +\tau_2))/2}} { (1-\sigma-\tau_2)/2 } \bigg) ^{ \frac{2(\sigma+\tau_2)}{(1+\sigma+\tau_2)}      } \bigg\}.
\end{align*}
From \eqref{rhon}, we obtain
\begin{eqnarray}\label{comp3 4}
\sup_{n}\{ \rho^n(\delta,t)\}\rightarrow 0\ \ \text{as}\ \ \delta \rightarrow 0.
\end{eqnarray}
Finally, adding (\ref{comp3 2}) and (\ref{comp3 4}), we obtain the desired result. This completes the proof of the second part of Lemma \ref{compactlemma}.
\end{proof}

Next, we turn to show the time equicontinuity of sequences $(g^n)_{n\in \mathbb{N}}$ and $(h^n)_{n\in \mathbb{N}} := (\mu^{-\sigma}g^n)_{n\in \mathbb{N}}$.

\subsection{Equicontinuity in time}\label{subs:equit}

Set $\Upsilon^n(\mu, t):=\mu^{-\zeta}g^n(\mu, t)$, for $\zeta =\{0, \sigma \}$. When $ \zeta =0$, $\Upsilon^n(\mu, t)=g^n(\mu, t)$  and $\zeta=\sigma$, $\Upsilon^n(\mu, t):=\mu^{-\zeta}g^n(\mu, t)$. Let $T>0$ be fixed, $\epsilon > 0$ and $\phi \in L^{\infty}(\mathbb{R}_{+})$ and consider $s,t\in [0,T]$
 with $t\geq s$. Next, fix $\lambda>1$ such that
\begin{eqnarray}\label{sequicontinuity1}
\frac{2 \mathcal{P}(T)}{\lambda^{1+\zeta}}< \frac{\epsilon}{2 \|\phi\|_{L^{\infty}}  },
\end{eqnarray}
where the constant $\mathcal{P}(T)$ is defined in Lemma \ref{LemmaEquibounded}. For each $n$, by Lemma \ref{LemmaEquibounded} and (\ref{sequicontinuity1}), we have
\begin{eqnarray}\label{sequicontinuity2}
\int_{\lambda}^{\infty}|\Upsilon^n(\mu, t)-\Upsilon^n(\mu, s)|d\mu \leq \frac{1}{\lambda}\int_{\lambda}^{\infty}
 \mu^{1-\zeta}[g^n(\mu, t)+g^n(\mu, s)]d\mu \leq \frac{2\mathcal{P}(T)}{\lambda^{1+\zeta}}< \frac{\epsilon}{2\|\phi\|_{L^{\infty}} }.
\end{eqnarray}
Multiplying by $\mu^{-\zeta} \phi(\mu)$ on both sides into (\ref{5trun cccecf}, integrating with respect to $\mu$ from $0$ to $\lambda$, using  Leibniz's rule and non-negativity of $g^n$, we simplify the following term as \begin{align}\label{sequicontinuity3}
 \bigg|\frac{d}{dt}  \int_{0}^{\lambda}  \phi(\mu) \Upsilon^n(\mu, t)d\mu \bigg|
\leq  \|\phi\|_{L^{\infty}} &\bigg[ \int_{0}^{\lambda} \mu^{-\zeta} \mathcal{B}^n_c(g^n)(\mu, t) d\mu
+ \int_{0}^{\lambda} \mu^{-\zeta} \mathcal{D}^n_{cb}(g^n)(\mu, t) d\mu \nonumber\\ 
&+ \int_{0}^{\lambda} \mu^{-\zeta} \mathcal{B}^n_b(g^n)(\mu, t) d\mu \bigg].
\end{align}
First, we estimate the following term by using Fubini's theorem, the transformation $\mu-\nu=\mu'$ $\&$ $\nu=\nu'$, \eqref{collisionKernels} and Lemma \ref{LemmaEquibounded}, as
\begin{eqnarray}\label{sequicontinuity31}
\int_{0}^{\lambda} \mu^{-\zeta} \mathcal{B}^n_c(g^n)(\mu, t) d\mu \leq \frac{1}{2} k (1+\lambda)^{2\omega}\mathcal{P}(T)^2.
\end{eqnarray}
Similarly, we evaluate the second integral of \eqref{sequicontinuity3}, as
\begin{eqnarray}\label{sequicontinuity32}
\int_{0}^{\lambda} \mu^{-\zeta} \mathcal{D}^n_{cb}(g^n)(\mu, t) d\mu \leq   k(1+\lambda)^{\omega} \mathcal{P}(T)^2.
\end{eqnarray}
Next, the last integral of \eqref{sequicontinuity3} is evaluated, by applying Fubini's theorem, \eqref{distribution1s}, \eqref{distribution4s}, as
\begin{align}\label{sequicontinuity33}
 \int_{0}^{\lambda} \mu^{-\zeta} \mathcal{B}^n_b(g^n)  (\mu, t) & d\mu  \leq \frac{1}{2} \int_{0}^{\lambda}\int_{\mu}^{n} \int_{0}^{\nu} \mu^{-\zeta}P(\mu|\nu -\tau;\tau) \Psi_n(\nu-\tau, \tau)g^n(\nu-\tau, t)g^n(\tau, t)d\tau d\nu d\mu \nonumber\\
 \leq & \frac{1}{2}  \eta(\zeta) \int_{0}^{\lambda}\int_{0}^{\nu}  \nu^{-\zeta}  \Psi_n(\nu-\tau, \tau)g^n(\nu-\tau, t)g^n(\tau, t)d\tau d\nu \nonumber\\
  +& \frac{1}{2} k'(\lambda) \int_{\lambda}^n \int_{0}^{\lambda} \int_{0}^{\nu} \mu^{-\zeta -\tau_2} \Psi_n(\nu-\tau, \tau)g^n(\nu-\tau, t)g^n(\tau, t)d\tau d\mu d\nu.
\end{align}
We estimate \eqref{sequicontinuity33} by applying Fubini's theorem, \eqref{collisionKernels}, replacing $\tau \rightarrow \nu$ $\&$ $\nu \rightarrow \mu$, the transformation $\mu -\nu=\mu'$ $\&$ $\nu=\nu'$, and Lemma \ref{LemmaEquibounded}, to have
 \begin{align}\label{sequicontinuity34}
 \int_{0}^{\lambda} \mu^{-\zeta} \mathcal{B}^n_b(g^n)(\mu, t) d\mu \leq & \frac{1}{2}  \eta(\zeta) k \int_{0}^{\lambda} \int_{0}^{\lambda- \mu}  (\mu+\nu)^{-\zeta}  \frac{(1+\mu)^{\omega} (1+\nu)^{\omega}}{(\mu+\nu)^{\omega}}  g^n(\mu, t)g^n(\nu, t)d\nu d\mu \nonumber\\
 & +\frac{1}{2} k'(\lambda) k\frac{\lambda^{1-\zeta -\tau_2}}{1-\zeta -\tau_2} \int_0^{\lambda}  \int_{\lambda-\nu}^{n-\nu}  \frac{(1+\mu)^{\omega} (1+\nu)^{\omega}}{(\mu+\nu)^{\omega}} g^n(\nu, t)g^n(\mu, t)d\mu d\nu \nonumber\\
  & +\frac{1}{2} k'(\lambda) k\frac{\lambda^{1-\zeta -\tau_2}}{1-\zeta -\tau_2} \int_{\lambda}^n  \int_{0}^{n-\nu}  \frac{(1+\mu)^{\omega} (1+\nu)^{\omega}}{(\mu+\nu)^{\omega}} g^n(\nu, t)g^n(\mu, t)d\mu d\nu \nonumber\\
\leq & \frac{1}{2}  \eta(\zeta) k (1+\lambda)^{\omega} \mathcal{P}(T)^2+k'(\lambda)  2^{\omega} k \frac{\lambda^{1-\zeta -\tau_2}}{1-\zeta -\tau_2} \mathcal{P}(T)^2.
\end{align}

Inserting the estimates on \eqref{sequicontinuity31}, \eqref{sequicontinuity32} and \eqref{sequicontinuity34} into (\ref{sequicontinuity3}) and then integrating with respect to time from $s$ to $t$, we obtain
\begin{eqnarray}\label{sequicontinuity4}
 \bigg| \int_{0}^{\lambda} \phi(\mu)[\Upsilon^{n}(\mu, t)-\Upsilon^{n}(\mu, s)]d\mu \bigg|  \leq  B_{2}(T) \| \phi \|_{L^{\infty} }(t-s),
\end{eqnarray}
where $B_{2}(T):=\frac{k}{2} (1+\lambda)^{\omega}\mathcal{P}(T)^2 [  (1+\lambda)^{\omega}+ 2 ]+ \frac{1}{2}  \eta(\zeta) k (1+\lambda)^{\omega} \mathcal{P}(T)^2+k'(\lambda) 2^{\omega}k \frac{\lambda^{1-\zeta -\tau_2}}{1-\zeta -\tau_2} \mathcal{P}(T)^2$.

Using (\ref{sequicontinuity2}) and (\ref{sequicontinuity4}), we get
\begin{equation}\label{equicontinuity5}
\left| \int_{0}^{\infty}\phi(\mu)[\Upsilon^n(\mu, t)-\Upsilon^n(\mu, s)] d\mu \right| \leq B_{2}(T)\ \|\phi\|_{L^{\infty}}\ (t-s) + \frac{\epsilon}{2}.
\end{equation}
Choose some suitably small number $\delta>0$ such that $t-s<\delta$. Then the estimate (\ref{equicontinuity5}) implies the equicontinuity of the family
$\{g^n(t), t\in[0,T]\}$ with respect to time variable $t$, in the topology $L^1(\mathbb{R}_{+})$. Then according to a refined version of the
\textit{Arzel\`{a}-Ascoli Theorem}, see \cite[Theorem 2.1]{Stewart:1989} or \textit{Arzel\`{a}-Ascoli Theorem} \cite[Appendix A8.5]{Ash:1972}, we conclude
 that there exist a subsequence (${\Upsilon^{n_k}}$) and a non-negative function $\Upsilon \in L^\infty((0,T);L^1(\mathbb{R}_{+}))$ such that
\begin{eqnarray*}
\lim_{n_k\to\infty} \sup_{t\in [0,T]}  \left| \int_0^\infty  [ \Upsilon^{n_k}(\mu, t) - \Upsilon(\mu, t)]\ \phi(\mu)\ d\mu \right| = 0, \label{vittel}
\end{eqnarray*}
for all $T>0$ and $\phi \in L^\infty(\mathbb{R}_{+})$. This implies that
\begin{eqnarray}\label{equicontinuityfinal}
  \Upsilon^{n_k}(t) \rightharpoonup \Upsilon(t)\ \ \text{in}\ \ L^1(\mathbb{R}_{+})\ \ \text{as}\ n \rightarrow \infty,
\end{eqnarray}
uniformly for all $t \in [0,T]$ to some $\Upsilon \in \mathcal{C}([0,T]; w-L^1(\mathbb{R}_{+}))$,  where  $\mathcal{C}([0, T]; w-L^1 (\mathbb{R}_{+}))$ is the space of all weakly continuous functions from $[0, T]$ to $L^1 (\mathbb{R}_{+})$. \\

 For $\zeta =0$ and $\zeta=\sigma $, we can conclude that there exist subsequences $(g^{n_k})_{k\in N}$ and $(h^{n_l})_{l\in N}$ such that
\begin{eqnarray}\label{equicontinuity6}
g^{n_k}(t) \rightharpoonup g(t)\ \ \text{in}\ \ L^1(\mathbb{R}_{+})\ \ \text{as}\ \ n_k \to \infty,
\end{eqnarray}
and
\begin{eqnarray}\label{equicontinuity7}
h^{n_l}(t) \rightharpoonup h(t)\ \ \text{in}\ \ L^1(\mathbb{R}_{+})\ \ \text{as}\ \ n_l \to \infty.
\end{eqnarray}

For any $ a>0, t\in [0,T], $ since we have $g^{n_k} \rightharpoonup g$, we thus obtain
\begin{align*}
\int_{1/a}^{a} \mu g(\mu, t)d\mu = \lim_{ n_k\to {\infty}} \int_{1/a}^{a} \mu g^{n_k}{(\mu, t)}d\mu \leq \|g_{\text{in}}\|_{\mathcal{S}} < \infty.
\end{align*}
and
\begin{align*}
\int_{1/a}^{a}\mu^{-\sigma}g(\mu, t)d\mu =
\lim_{ n_k\to {\infty}} \int_{1/a}^{a} \mu^{-\sigma} g^{n_k}{(\mu, t)}d\mu \leq \|g_{\text{in}}\|_{\mathcal{S}} < \infty.
\end{align*}
Using (\ref{strunc mass1}), the non-negativity of each $g^{n_k}$ and $g$, then as $ a \to \infty $ implies that $\int_0^{\infty}\mu g(\mu, t) d\mu \le \int_0^{\infty}\mu g_{\text{in}}(\mu) d\mu$ and $g \in \mathcal{S}^+ $.\\

Next, to show that the limit function $g$  constructed in (\ref{equicontinuity6}) is actually a weak solution to (\ref{sccecfe})--(\ref{sin1}) in an appropriate sense.
\subsection{Integral convergence}

In the following lemma, we wish to show the truncated integrals on the right-hand side of \eqref{5trun cccecf} convergence weakly to the original integrals on the right-hand of \eqref{sccecfe}.
\begin{lem}\label{Weakconvergencelemmasccbe}
Let $(g^{n})_{n\in \mathbb{N}}$ be a bounded sequence in $\mathcal{S}^+$ and $g\in \mathcal{S}^{+}$, where $\|g^n\|_{\mathcal{S}}  \leq \mathcal{P}(T)$ and $g^n\rightharpoonup g$ in $L^1(\mathbb{R}_{+} )$ as $n\to \infty $. Then, for each $a> 0$, we have
\begin{equation}\label{weakconvergences}
 \mathcal{B}_{c}^n(g^n) \rightharpoonup  \mathcal{B}_{c}(g),\ \mathcal{D}_{cb}^n(g^n)  \rightharpoonup  \mathcal{D}_{cb}(g)\  \text{and}\ \ \mathcal{B}_{b}^n(g^n)  \rightharpoonup  \mathcal{B}_{b}(g)\   \text{in} \  L^1(0, a)\  \text{as}\  n\to \infty.
\end{equation}
\end{lem}

\begin{proof} Let $0<a<n$ and $\mu \in (0, a)$. Suppose $\phi$ belongs to $L^\infty(0, a)$ with compact support included in $(0, a)$.
Using the similar ways with slight modifications as in  Camejo and Warnecke \cite{Camejo:2015II}, it can be easily shown that the first two terms such that $ \mathcal{B}_{c}^n(g^n) \rightharpoonup  \mathcal{B}_{c}(g)$ and $\mathcal{D}_{cb}^n(g^n)  \rightharpoonup  \mathcal{D}_{cb}(g)$.\\

Next, in order to show  $ \mathcal{B}_{b}^n(g^n)  \rightharpoonup   \mathcal{B}_{b}(g)\ \text{in} \  L^1(0, a)\ \text{as}\  n\to \infty$, it is sufficient to prove
\begin{align}\label{Weakconvegences1}
\bigg| \int_{0}^{a}  \phi(\mu) [\mathcal{B}_b^n(g^n)(\mu, t)-\mathcal{B}_b(g)(\mu,t)] d\mu\bigg|  \to 0,\ \text{as}\ n \to \infty,\ \ \text{for}\ \phi \in L^\infty(0, a).
 \end{align}
Let us first simplify the following integral by using triangle inequality, as
\begin{align}\label{Weakconvegences2}
\bigg| \int_{0}^{a}  \phi(\mu) [&\mathcal{B}_b^n(g^n)(\mu, t)-\mathcal{B}_b(g)(\mu,t)] d\mu\bigg|\nonumber\\
 \leq & \frac{1}{2} \bigg|\int_{0}^{a} \int_{\mu}^{b}\int_{0}^{\nu} \phi(\mu) P(\mu|\nu-\tau;\tau) E'(\nu-\tau,\tau) \Psi(\nu-\tau,\tau) \nonumber\\ &~~~~~~~~~~~\times [g^n(\nu-\tau, t)g^n(\tau, t)-g(\nu-\tau, t)g(\tau, t)]d\tau d\nu d\mu \bigg| \nonumber\\
 &+\frac{1}{2}\bigg|\int_{0}^{a} \int_{b}^{\infty}\int_{0}^{\nu} \phi(\mu) P(\mu|\nu-\tau;\tau) E'(\nu-\tau,\tau) \Psi(\nu-\tau,\tau)\nonumber\\
&~~~~~~~~~~\times [g^n(\nu-\tau, t)g^n(\tau, t)-g(\nu-\tau, t)g(\tau, t)]d\tau d\nu d\mu \bigg|,
:=\mathcal{I}_1^n+\mathcal{I}_2^n,
 \end{align}
where $\mathcal{I}_1^n$ and $\mathcal{I}_2^n$ denote the first and the second terms, respectively, on the right-hand side to \eqref{Weakconvegences2} and we choose $b$ large enough for a given $\epsilon >0$ such that $n>b> a$ with
\begin{align}\label{Weakconvegences3}
\frac{k k'(a) a^{(1-\tau_2)}} {2(1-\tau_2)} \| \phi \|_{L^{\infty}(0, a)}[  \mathcal{P}(T)^2+\|g\|^2_{\mathcal{S}}  ]\frac{1} {b^{\sigma}} < \frac{\epsilon}{2}.
\end{align}

Here, let us first consider $\mathcal{I}_1^n$ by using Fubini's theorem, as \begin{align}\label{Weakconvegences4}
\mathcal{I}_1^n= 
 =&  \frac{1}{2} \bigg|  \bigg\{\int_0^a \int_0^{\nu} \int_0^{\nu} + \int_a^b \int_0^a \int_0^{\nu} \bigg\}  \phi(\mu) P(\mu|\nu-\tau; \tau) E'(\nu-\tau, \tau)\nonumber\\
 & ~~~   \times \Psi(\nu-\tau, \tau) g^n(\tau, t)[g^n(\nu-\tau, t)-g(\nu-\tau, t)] d\tau d\mu d\nu     \nonumber\\
  &+  \frac{1}{2}   \bigg\{ \int_0^a \int_0^{\nu} \int_0^{\nu} +\int_a^b \int_0^a \int_0^{\nu}  \bigg\}  \phi(\mu) P(\mu|\nu-\tau; \tau)  E'(\nu-\tau, \tau) \nonumber\\
 & ~~~   \times   \Psi(\nu-\tau, \tau)g(\nu-\tau, t)   [g^n(\tau, t)-g(\tau, t)] d\tau d\mu d\nu  \bigg| \leq  \sum_{i=1}^4  |H_i^n |,
\end{align}
where $|H_i^n|$, for $i=1,2,3,4$ are the modulus of first, second, third and fourth integrals to \eqref{Weakconvegences4}, respectively. Thanks to triangle inequality. Next, we evaluated $|H_i^n|$, for $i=1,2,3,4$ individually. First, we consider $|H_1^n|$ by using Fubini's theorem and the transformation $\nu-\tau=\nu'$ $\&$ $\tau=\tau'$, as
\begin{align*}
|H_1^n| 
 = &  \frac{1}{2} \bigg| \int_0^a \int_0^{a-\nu} \int_0^{\nu}   \phi(\mu) P(\mu|\nu; \tau)  E'(\nu, \tau) \Psi(\nu, \tau)
  g^n(\tau, t) [g^n(\nu, t)-g(\nu, t)] d\mu d\tau d\nu \bigg|.
\end{align*}

We estimate the following term by using \eqref{collisionKernels}, \eqref{sTotal particles} and Lemma \ref{LemmaEquibounded}, for each $\nu \in (0, a]$, as
\begin{align}\label{sH111}
 \frac{1}{2} \bigg| \int_0^{a-\nu} \int_0^{\nu} &  \phi(\mu) P(\mu|\nu; \tau) E'(\nu, \tau)  \Psi(\nu, \tau)g^n(\tau, t)  d\mu d\tau \bigg| \nonumber\\
  \leq &  \frac{k}{2} \|\phi\|_{L^{\infty}(0, a)} N (1+\nu)^{\omega}   \int_0^{a}  \frac{(1+\tau)^{\omega}}{\tau^{\sigma}} g^n(\tau, t)  d\tau  \nonumber\\
 \leq &  \frac{k}{2} \|\phi\|_{L^{\infty}(0, a)} N(1+\nu)^{\omega} (1+a)^{\omega}   \mathcal{P}(T) \in L^{\infty} (0,a).
\end{align}

From (\ref{sH111}) and weak convergence of $g^n \rightharpoonup g$ in $L^1(\mathbb{R}_{+})$, gives
\begin{align}\label{sH1111}
|H_1^n| \to 0,\ \ \ \text{as}\ \ n \to \infty.
\end{align}

Simplifying $|H_2^n|$ by applying Fubini's theorem twice, $\nu-\tau=\nu'$ $\&$ $\tau=\tau'$ and using symmetry of $P$, $E'$ $\&$ $\Psi$, as
\begin{align}\label{sH2}
|H_2^n| 
  = &  \frac{1}{2} \bigg| \int_0^a  \int_{a-\tau}^{b-\tau}  \int_0^{a} \phi(\mu) P(\mu|\nu; \tau)E'(\nu, \tau) \Psi(\nu, \tau)
  g^n(\tau, t) [g^n(\nu, t)-g(\nu, t)] d\mu d\nu d\tau     \nonumber\\
   & + \frac{1}{2} \int_a^b  \int_0^{b-\tau}  \int_0^{a} \phi(\mu) P(\mu|\nu; \tau)E'(\nu, \tau) \Psi(\nu, \tau)
  g^n(\tau, t) [g^n(\nu, t)-g(\nu, t)] d\mu d\nu d\tau \bigg|    \nonumber\\
 \le &  \frac{1}{2} \bigg| \int_0^a  \int_{a-\nu}^{b-\nu}  \int_0^{a} \phi(\mu) P(\mu|\nu; \tau)E'(\nu, \tau) \Psi(\nu, \tau)
  g^n(\tau, t) [g^n(\nu, t)-g(\nu, t)] d\mu d\tau d\nu  \bigg|   \nonumber\\
    + & \frac{1}{2} \bigg| \int_a^b  \int_0^{b-\nu}  \int_0^{a} \phi(\mu) P(\mu|\nu; \tau)E'(\nu, \tau) \Psi(\nu, \tau)
  g^n(\tau, t) [g^n(\nu, t)-g(\nu, t)] d\mu d\tau d\nu \bigg|.
\end{align}
Next, by using \eqref{collisionKernels}, \eqref{distribution4s} and Lemma \ref{LemmaEquibounded}, we estimate the following term for $\nu \in (0, a)$, as
\begin{align}\label{sH21}
  \frac{1}{2} \bigg|\int_{a-\nu}^{b-\nu}  \int_0^{a} & \phi(\mu) P(\mu|\nu; \tau)E'(\nu, \tau) \Psi(\nu, \tau)
  g^n(\tau, t)  d\mu d\tau \bigg| \nonumber\\
   \leq  &  \frac{1}{2} k \|\phi \|_{L^{\infty}(0, a)} k'(a) \frac{ a^{1-\tau_2}}{1-\tau_2} (1+\nu)^{\omega} (1+b)^{\omega} \mathcal{P}(T) \in L^{\infty}(0, a).
\end{align}
Similarly, by using \eqref{collisionKernels}, \eqref{distribution4s} and Lemma \ref{LemmaEquibounded}, we evaluate the following term for $\nu \in (a, b)$, as
\begin{align}\label{sH22}
  \frac{1}{2} \bigg|\int_{0}^{b-\nu}  \int_0^{a} & \phi(\mu) P(\mu|\nu; \tau)E'(\nu, \tau) \Psi(\nu, \tau)
  g^n(\tau, t)  d\mu d\tau \bigg| \nonumber\\
   \leq  &  \frac{1}{2} k \|\phi \|_{L^{\infty}(0, a)} k'(a) \frac{ a^{1-\tau_2}}{1-\tau_2} (1+\nu)^{\omega} (1+b)^{\omega} \mathcal{P}(T) \in L^{\infty}(a, b).
\end{align}

Hence, from (\ref{sH21}), (\ref{sH22}) and the weak convergence of $g^n \rightharpoonup g$ in $L^1(\mathbb{R}_{+})$, we have
\begin{align}\label{sH222}
|H_2^n| \to 0,\ \ \ \text{as}\ \ n \to \infty.
\end{align}

Let us now simplify $|H_3^n|$, by using Fubini's theorem twice, and $\nu-\tau=\nu'$ $\&$ $\tau=\tau'$, as
\begin{align}\label{H3}
|H_3^n| 
   = &  \frac{1}{2} \bigg| \int_0^a \int_0^{a-\tau} \int_0^{\nu}   \phi(\mu) P(\mu|\nu; \tau)E'(\nu, \tau) \Psi(\nu, \tau)
  g(\nu, t) [g^n(\tau, t)-g(\tau, t)] d\mu d\nu d\tau \bigg|.
\end{align}

Next, we estimate the following integral \eqref{collisionKernels}, \eqref{sTotal particles} and Lemma \ref{LemmaEquibounded}, we have
\begin{align}\label{H33}
 \frac{1}{2} \bigg| \int_0^{a-\tau} \int_0^{\nu} &  \phi(\mu) P(\mu|\nu; \tau)E'(\nu, \tau) \Psi(\nu, \tau) g(\nu, t)  d\mu d\nu  \bigg| \nonumber\\
 \leq & k \frac{1}{2} \|\phi\|_{L^{\infty}(0, a)} N (1+\tau)^{\omega}  \int_0^{a} \frac{ (1+\nu)^{\omega} }{ \nu^{\sigma}} g(\nu, t) d\nu\nonumber\\
 \leq & \frac{1}{2} k \|\phi \|_{L^{\infty}(0, a)} N  (1+\tau)^{\omega} (1+a)^{\omega} \|g\|_{\mathcal{S}}   \in L^{\infty}(0, a).
\end{align}

From (\ref{H33}) and weak convergence of $g^n$ in $L^1(\mathbb{R}_+)$, confirms that
\begin{align}\label{H333}
|H_3^n| \to 0,\ \ \ \text{as}\ \ n \to \infty.
\end{align}

Simplifying $|H_4^n|$ by using Fubini's theorem twice, and  $\nu-\tau=\nu'$ $\&$ $\tau=\tau'$, we have
\begin{align}\label{H4}
|H_4^n| = 
  \leq &  \frac{1}{2} \bigg| \int_0^a  \int_{a-\tau}^{b-\tau}  \int_0^{a} \phi(\mu) P(\mu|\nu; \tau)E'(\nu, \tau) \Psi(\nu, \tau)
  g^n(\nu, t) [g^n(\tau, t)-g(\tau, t)] d\mu d\nu d\tau  \bigg|   \nonumber\\
   + & \frac{1}{2} \bigg| \int_a^b  \int_0^{b-\tau}  \int_0^{a} \phi(\mu) P(\mu|\nu; \tau)E'(\nu, \tau) \Psi(\nu, \tau)
  g^n(\nu, t) [g^n(\tau, t)-g(\tau, t)] d\mu d\nu d\tau \bigg|.
\end{align}

Similar to $|H_2^n|$, we can easily show that
\begin{align}\label{H444}
|H_4^n| \to 0,\ \ \ \text{as}\ \ n \to \infty.
\end{align}
Using (\ref{sH1111}), (\ref{sH222}), (\ref{H333}) and (\ref{H444}) into \eqref{Weakconvegences4} and thus, we get
\begin{align}\label{Weakconvegences21}
\frac{1}{2} \bigg| \int_0^a\int_\mu^b \int_0^{\nu} &  \phi(\mu)P(\mu|\nu-\tau; \tau)  E'(\nu-\tau, \tau) \Psi(\nu-\tau, \tau)\nonumber\\
 & \times [g^n(\nu-\tau, t)g^n(\tau, t)-g(\nu-\tau, t)g(\tau, t)]d\tau d\nu d\mu \bigg|  \to 0\ \ \text{as}\ n\to \infty.
\end{align}

Next, $\mathcal{I}_2^n$ can be simplified by using Fubini's theorem twice and the transformation $\nu-\tau=\nu'$ $\&$ $\tau=\tau'$, we obtain
\begin{align}\label{Weakconvegences22}
\mathcal{I}_2^n
 = &\frac{1}{2} \bigg| \int_{0}^{a}  \int_{0}^{b}\int_{b}^{\infty} \phi(\mu) P(\mu|\nu-\tau; \tau) E'(\nu-\tau, \tau) \Psi(\nu-\tau, \tau) \nonumber\\
 & ~~~~~~~~~~~~~~~\times [g^n(\nu-\tau, t)g^n(\tau, t)-g(\nu-\tau, t)g(\tau, t)]d\nu d\tau d\mu  \nonumber\\
 &+ \int_{0}^{a}  \int_{b}^{\infty}\int_{\tau}^{\infty} \phi(\mu) P(\mu|\nu-\tau; \tau) E'(\nu-\tau, \tau) \Psi(\nu-\tau, \tau)\nonumber\\
 & ~~~~~~~~~~~~~~\times [g^n(\nu-\tau, t)g^n(\tau, t)-g(\nu-\tau, t)g(\tau, t)]d\nu d\tau d\mu \bigg| \nonumber\\
  \le & \frac{1}{2} \bigg|  \int_{0}^{b}\int_{b-\tau}^{n} \int_{0}^{a} \phi(\mu) P(\mu| \nu; \tau) E'(\nu, \tau) \Psi(\nu, \tau) [g^n(\nu, t)g^n(\tau, t)-g(\nu, t)g(\tau, t)]d\mu  d\nu d\tau \bigg| \nonumber\\
  &+   \frac{1}{2} \bigg|   \int_{b}^{\infty}\int_{0}^{n} \int_{0}^{a}  \phi(\mu) P(\mu| \nu; \tau) E'(\nu, \tau) \Psi(\nu, \tau) [g^n(\nu, t)g^n(\tau, t)-g(\nu, t)g(\tau, t)]d\mu  d\nu d\tau \bigg| \nonumber\\
   &+ \frac{1}{2} \bigg|  \int_{0}^{\infty}\int_{n}^{\infty} \int_{0}^{a} \phi(\mu) P(\mu| \nu; \tau) E'(\nu, \tau) \Psi(\nu, \tau) [g^n(\nu, t)g^n(\tau, t)-g(\nu, t)g(\tau, t)]d\mu  d\nu d\tau \bigg|\nonumber\\
   =:& \sum_{i=1}^3 N_i,
\end{align}
where $N_i's$, for $i=1, 2, 3$ are the first, the second and the third terms, respectively, on the right-hand side to \eqref{Weakconvegences22}. Now, $N_1$ can be simplified by interchanging $\nu$ to $\tau$ and using the symmetry of $P$, $E'$ $\&$ $\Psi$, to have \begin{align}\label{Weakconvegences23}
N_1 
   \le & \frac{1}{2} \bigg| \bigg\{  \int_{0}^{b}\int_{b-\nu}^{b} \int_{0}^{a} +\int_{b}^{n}\int_{0}^{b} \int_{0}^{a} \bigg\} \phi(\mu)P(\mu| \nu; \tau) E'(\nu, \tau) \Psi(\nu, \tau)  g(\tau, t)[g^n(\nu, t)-g(\nu, t)]  d\mu d\tau d\nu  \bigg| \nonumber\\
 & + \frac{1}{2} \bigg| \bigg\{  \int_{0}^{b}\int_{b-\tau}^{b} \int_{0}^{a} +\int_{0}^{b}\int_{b}^{n} \int_{0}^{a} \bigg\} \phi(\mu)P(\mu| \nu; \tau) E'(\nu, \tau) \Psi(\nu, \tau) \nonumber\\
 & ~~\times g^n(\nu, t)[g^n(\tau, t)-g(\tau, t)] d\mu d\nu d\tau  \bigg|
\end{align}
By using \eqref{distribution4s}, \eqref{collisionKernels}, Lemma \ref{LemmaEquibounded} and $g^n \rightharpoonup g$ in $L^1(\mathbb{R}_{+})$, we can easily seen that \begin{align}\label{Weakconvegences24}
  N_1 \to 0, \ \ \text{as}\ \ n\to \infty.
\end{align}
Next, we estimate $N_2$ by using  \eqref{distribution4s}, \eqref{collisionKernels}, Lemma \ref{LemmaEquibounded} and \eqref{Weakconvegences3}, as \begin{align}\label{Weakconvegences25}
N_2    
 \le & \frac{1}{2} \| \phi \|_{L^{\infty}(0, a)}   k k'(a)  \int_{b}^{\infty}\int_{0}^{n} \int_{0}^{a}  \mu^{-\tau_2}   \frac{(1+\nu)^{\omega} (1+\tau)^{\omega}}{(\nu+\tau)^{\sigma}}   [g^n(\nu, t)g^n(\tau, t)+g(\nu, t)g(\tau, t)]d\mu  d\nu d\tau  \nonumber\\
 \le &   \frac{1}{2} \frac{k k'(a) a^{1-\tau_2}} {1-\tau_2} \| \phi \|_{L^{\infty}(0, a)}   \bigg[ \mathcal{P}(T)^2+\|g\|^2_{\mathcal{S}} \bigg]\frac{1}{b^{\sigma}} < \epsilon.
\end{align}

Now consider $N_3$, by using \eqref{collisionKernels}, \eqref{distribution4s}, $g^n=0$ as $n>\mu$ and Fubini's theorem, as
\begin{align*}
N_3 
 \le & \frac{1}{2} k k'(a) \| \phi \|_{L^{\infty}(0, a)} \int_{0}^{\infty} \int_{1}^{\infty}\int_{0}^{a}   \mu^{-\tau_2} \frac{(1+\nu)^{\omega} (1+\tau)^{\omega}}{(\nu+\tau)^{\sigma}}  g(\nu, t)g(\tau, t) d\mu  d\nu d\tau \nonumber\\
\le & \frac{1}{2} \frac{k k'(a) a^{1-\tau_2}} {1-\tau_2} \| \phi \|_{L^{\infty}(0, a)}  \|g\|^2_{\mathcal{S}} < \infty.
\end{align*}
Hence, $N_3$ is uniformly bounded for each $n$. Then, by Lebesgue dominated convergence theorem
\begin{align}\label{Weakconvegences26}
N_3 
  \to 0,  \ \ \text{as}\ \ n\to \infty.
\end{align}

Now, using (\ref{Weakconvegences24}), (\ref{Weakconvegences25}) and (\ref{Weakconvegences26}) into (\ref{Weakconvegences22}), we obtain
\begin{align}\label{Weakconvegences27}
\frac{1}{2} \bigg| \int_{0}^{a} & \int_{b}^{\infty}\int_{0}^{\nu} \phi(\mu)P(\mu|\nu-\tau; \tau) E'(\nu-\tau, \tau) \Psi(\nu-\tau, \tau)\nonumber\\
 & ~~~~~~~~~~~~~~~\times [g^n(\nu-\tau, t)g^n(\tau, t)-g(\nu-\tau, t)g(\tau, t)]d\tau d\nu d\mu \bigg| \to 0\ \text{as}\ n \to \infty.
 \end{align}
Applying \eqref{Weakconvegences2} and \eqref{Weakconvegences26} into \eqref{Weakconvegences27}, we conclude that
\begin{align*}
 \lim_{n \rightarrow \infty }{ \mathcal{B}^n_b{(g^n)}\rightharpoonup \mathcal{B}_b{(g)} },
\end{align*}

 which completes the proof of Lemma \ref{Weakconvergencelemmasccbe}.
\end{proof}

Now, this is the right place to prove the Theorem \ref{existmain theorem1} with the help of above results.

\begin{proof}
 \textit{of Theorem \ref{existmain theorem1}}: Fix $a \in (0, n_k)$, $T>0$ 
and $ \phi \in \L^{\infty}(0, a)$ then from Lemma \ref{Weakconvergencelemmasccbe}, we have for each $ s \in [0,t]$
\begin{align}\label{main convergence1}
\int_{0}^{a}\phi(\mu)[ (\mathcal{B}_c^{n_k}-\mathcal{D}_{cb}^{n_k}+\mathcal{B}_b^{n_k}) (g^{n_k})(\mu,s)-(\mathcal{B}_c-\mathcal{D}_{cb}+\mathcal{B}_b  )(g)(\mu,s)]d\mu \to 0 \ \ \text{as} \ \ n_k\to \infty.
\end{align}
  Again, using by using repeated application of Fubini's theorem and using $\mu-\nu=\mu'$, $\&$ $\nu=\nu'$ \eqref{collisionKernels}, \eqref{distribution4s}, (\ref{sTotal particles}) and Lemma \ref{LemmaEquibounded}, there exists a positive constant $C(a, T)$ such that

\begin{align}\label{exist domin}
\bigg|\int_{0}^{a} &\phi(\mu)[ (\mathcal{B}_c^{n_k}-\mathcal{D}_{cb}^{n_k}+\mathcal{B}_b^{n_k}) (g^{n_k})(\mu,s) d \mu \bigg| \le C(a, T) \|\phi\|_{L^{\infty}(0, a)},
\end{align}
where
\begin{align*}
C(a, T) :=  \bigg[1/2k (3 \mathcal{P}(T)^2+5\|g\|^2_{\mathcal{S}}) + k N[ \mathcal{P}(T)^2 +\|g\|^2_{\mathcal{S}} ] + k k'(a) \frac{a^{1-\tau_2}}{1-\tau_2} \|g\|^2_{\mathcal{S}} \bigg].
\end{align*}
 One can see that the left-hand side of (\ref{exist domin}) is in $L^{1}(0, t)$, then from (\ref{main convergence1}), (\ref{exist domin}) and the Lebesgue dominated convergence theorem, we obtain
\begin{align}\label{converge1}
\int_0^t\int_{0}^{a} \phi(\mu)[ (\mathcal{B}_c^{n_k}-\mathcal{D}_{cb}^{n_k}+\mathcal{B}_b^{n_k}) (g^{n_k})(\mu,s)-(\mathcal{B}_c-\mathcal{D}_{cb}+\mathcal{B}_b  )(g)(\mu,s)] d\mu ds \to 0,
\end{align}
as $n_k \to \infty$. Since $\phi$ is arbitrary and (\ref{converge1}) holds for $ \phi \in L^{\infty}(0, a)$ as $n_k \to \infty$, hence, by applying Fubini's theorem, we get
\begin{align}\label{exist last1}
\int_{0}^{t}(\mathcal{B}_c^{n_k}-\mathcal{D}_{cb}^{n_k}+\mathcal{B}_b^{n_k}) (g^{n_k})(\mu,s) ds\rightharpoonup  \int_{0}^{t} (\mathcal{B}_c-\mathcal{D}_{cb}+\mathcal{B}_b  )(g)(\mu,s)ds\  \text{ in}\ L^{1}(0, a).
\end{align}
Then by the definition of  $(\mathcal{B}_c^{n_k}-\mathcal{D}_{cb}^{n_k}+\mathcal{B}_b^{n_k})$  we obtain
\begin{align}\label{exist last2}
g^{n_k}(\mu,t)= \int_{0}^{t} (\mathcal{B}_c^{n_k}-\mathcal{D}_{cb}^{n_k}+\mathcal{B}_b^{n_k}) (g^{n_k})(\mu, s)ds+ g^{n}_{in}(\mu),\ \ \ \text{for}\ t \in [0,T].
\end{align}
Next, using (\ref{exist last1}), \eqref{equicontinuity6} and (\ref{exist last2}), we thus obtain
\begin{align*}
\int_{0}^{a}\phi(\mu)g(\mu,t)d\mu = \int_{0}^{t} \int_{0}^{a} \phi(\mu)(\mathcal{B}_c-\mathcal{D}_{cb}+\mathcal{B}_b  )(g)(\mu, s)d\mu s+\int_{0}^{a}\phi(\mu)g_{in}(\mu)d\mu,
\end{align*}
for any $\phi \in L^{\infty}(0, a)$.  Hence, for the arbitrariness of $T$, $a$, the uniqueness of limit and for all $\phi \in L^{\infty}(0, a)$, we have $g(\mu, t)$ is a solution to (\ref{sccecfe})--(\ref{sin1}). This implies that for almost any $\mu \in (0, a)$, we have
\begin{align*}
g(\mu,t)= \int_{0}^{t}(\mathcal{B}_c-\mathcal{D}_{cb}+\mathcal{B}_b  )(g) (\mu, s)ds+ g_{in}(\mu).
\end{align*}
This completes the proof of the existence Theorem \ref{existmain theorem1}. \end{proof}

 In the next section, the uniqueness of weak solutions to (\ref{sccecfe})--(\ref{sin1}) is investigated under additional growth condition \eqref{collisionKerneluniques} on collision kernel $\Psi$. The proof of the uniqueness result is motivated from \cite{Stewart:1990}.

\section{Uniqueness of weak solutions}
\begin{proof}
\textit {of Theorem \ref{uniquemain theorem1}:}
Let $g$ and $h$ be two weak solutions to \eqref{sccecfe}--\eqref{sin1} on $[0, \infty)$, with $ g_{in}(\mu)=h_{in}(\mu)$. Set $ Z:=g-h$. For $n= 1,2,3,\cdots,$ we define
\begin{eqnarray}\label{Uniqueness1}
\Xi^{n}(t):= \int_{0}^{n}(1+\mu^{-\sigma}) |Z(\mu,t)|d\mu=\int_{0}^{n}(1+\mu^{-\sigma}) \mbox{sgn}(Z(\mu,t)) [g(\mu,t)-h(\mu,t)]d\mu.
\end{eqnarray}
Thanks to the property of the signum function. Next, substituting the value of $g(\mu,t)-h(\mu,t)$ by using Definition \ref{sdef1} $(iii)$ into \eqref{Uniqueness1} and simplifying it further by applying the following identity
\begin{eqnarray*}\label{Uniqueness2}
g(\mu,s)g(\nu,s)-h(\mu,s)h(\nu,s)= g(\mu,s)Z(\nu,s)+h(\nu,s)Z(\mu,s),
\end{eqnarray*}
 Fubini's theorem, symmetry of $\Psi$, the transformation $\mu-\nu=\mu'$ $\&$ $\nu=\nu'$ and $\nu-\tau=\nu'$ $\&$ $\tau=\tau'$, we have
\begin{align}\label{Uniqueness3}
\Xi^{n}(t)=&  \frac{1}{2} \int_{0}^{t}\int_{0}^{n}\int_{0}^{n-\mu} \bigg[[1+(\mu+\nu)^{-\sigma}] \mbox{sgn}(Z(\mu+\nu, s))E(\mu, \nu) -(1+\mu^{-\sigma})\mbox{sgn}(Z(\mu, s)) \nonumber\\
& -(1+\nu^{-\sigma})\mbox{sgn}(Z(\nu, s)) \bigg] \Psi(\mu, \nu)\{g(\mu, s)Z(\nu, s)+h(\nu, s)Z(\mu, s)\}d\nu d\mu ds\nonumber\\
& -\int_{0}^{t}\int_{0}^{n}\int_{n-\mu}^{\infty}(1+\mu^{-\sigma}) \mbox{sgn}(Z(\mu,s))  \Psi(\mu, \nu) \{g(\mu,s)Z(\nu, s)+h(\nu, s)Z(\mu, s)\} d\nu d\mu ds\nonumber\\
&+ \frac{1}{2} \int_{0}^{t}\int_{0}^{n}\int_{0}^{n-\nu}\int_{0}^{\nu+\tau}(1+\mu^{-\sigma}) \mbox{sgn}(Z(\mu, s)) P(\mu|\nu;\tau) E'(\nu, \tau) \Psi(\nu, \tau)\nonumber\\
&\hspace{2cm}\times \{g(\nu, s)Z(\tau, s)+h(\tau, s)Z(\nu, s)\} d\mu d\tau d\nu ds\nonumber\\
&+ \frac{1}{2} \int_{0}^{t}\int_{0}^{n}\int_{n-\nu}^{\infty}\int_{0}^{n}(1+\mu^{-\sigma}) \mbox{sgn}(Z(\mu, s))P(\mu|\nu;\tau) E'(\nu, \tau)] \Psi(\nu, \tau)\nonumber\\
&\hspace{2cm}\times \{ g(\nu, s)Z(\tau, s)+h(\tau, s)Z(\nu, s)\} d\mu d\tau d\nu ds\nonumber\\
&+ \frac{1}{2} \int_{0}^{t}\int_{n}^{\infty}\int_{0}^{\infty}\int_{0}^{n}(1+\mu^{-\sigma}) \mbox{sgn}(Z(\mu,s))P(\mu|\nu;\tau) E'(\nu, \tau)] \Psi(\nu, \tau)\nonumber\\
&\hspace{2cm}\times \{g(\nu, s)Z(\tau, s)+h(\tau, s)Z(\nu, s)\}d\mu d\tau d\nu ds.
\end{align}
Next, let us define the term $Q$ by
\begin{align*}
Q(\mu, \nu, s):=& [1+(\mu+\nu)^{-\sigma}] \mbox{sgn}(Z(\mu+\nu, s))E(\mu, \nu) -(1+\mu^{-\sigma})\mbox{sgn}(Z(\mu,s))\\
&-(1+\nu^{-\sigma})\mbox{sgn}(Z(\nu,s)).
\end{align*}

Using the definition of $Q$ and the properties of signum function into \eqref{Uniqueness3}, we obtain
\begin{align}\label{Uniqueness4}
\Xi^{n}(t)\le &  \frac{1}{2} \int_{0}^{t}\int_{0}^{n}\int_{0}^{n-\mu}  Q(\mu, \nu, s) \Psi(\mu, \nu) g(\mu, s)Z(\nu, s) d\nu d\mu ds\nonumber\\
&+\frac{1}{2} \int_{0}^{t}\int_{0}^{n}\int_{0}^{n-\mu}  Q(\mu, \nu, s) \Psi(\mu, \nu) h(\nu, s)Z(\mu, s) d\nu d\mu ds\nonumber\\
& +\int_{0}^{t}\int_{0}^{n}\int_{n-\mu}^{\infty}(1+\mu^{-\sigma})  \Psi(\mu, \nu) g(\mu, s) |Z(\nu,s)| d\nu d\mu ds\nonumber\\
& -\int_{0}^{t}\int_{0}^{n}\int_{n-\mu}^{\infty}(1+\mu^{-\sigma})   \Psi(\mu, \nu) h(\nu,s) |Z(\mu,s)| d\nu d\mu ds\nonumber\\
&+ \frac{1}{2} \int_{0}^{t}\int_{0}^{n}\int_{0}^{\infty}\int_{0}^{\nu+\tau}(1+\mu^{-\sigma})  P(\mu|\nu;\tau)  \Psi(\nu, \tau)\nonumber\\
&\hspace{2cm}\times \{g(\nu, s) |Z(\tau, s)|+h(\tau, s) |Z(\nu, s)|\}d\mu d\tau d\nu ds\nonumber\\
&+ \frac{1}{2} \int_{0}^{t}\int_{n}^{\infty}\int_{0}^{\infty}\int_{0}^{\nu+\tau}(1+\mu^{-\sigma}) P(\mu|\nu;\tau)  \Psi(\nu, \tau)\nonumber\\
&\hspace{2cm}\times \{g(\nu, s) |Z(\tau, s)|+h(\tau, s)|Z(\nu, s)|\} d\mu d\tau d\nu ds.
\end{align}
Due to the non-negativity of the fourth integral on the right-hand side to \eqref{Uniqueness4}, \eqref{sTotal particles} and $r=\sigma$ into \eqref{distribution1s}, we estimate  \eqref{Uniqueness4}, as \begin{align}\label{Uniqueness5}
\Xi^{n}(t)\le &  \frac{1}{2} \int_{0}^{t}\int_{0}^{n}\int_{0}^{n-\mu}  Q(\mu, \nu, s) \Psi(\mu, \nu) g(\mu, s)Z(\nu, s) d\nu d\mu ds\nonumber\\
&+\frac{1}{2} \int_{0}^{t}\int_{0}^{n}\int_{0}^{n-\mu}  Q(\mu, \nu, s) \Psi(\mu, \nu) h(\nu, s)Z(\mu, s) d\nu d\mu ds\nonumber\\
& +\int_{0}^{t}\int_{0}^{n}\int_{n-\mu}^{\infty}(1+\mu^{-\sigma})  \Psi(\mu, \nu) g(\mu,s) |Z(\nu,s)| d\nu d\mu ds\nonumber\\
&+ \frac{1}{2} \int_{0}^{t}\int_{0}^{n}\int_{0}^{n} [N+\eta(\sigma)(\nu+\tau)^{-\sigma}]  \Psi(\nu, \tau)g(\nu, s) |Z(\tau, s)|  d\tau d\nu ds\nonumber\\
&+ \frac{1}{2} \int_{0}^{t}\int_{0}^{n}\int_{n}^{\infty} [N+\eta(\sigma)(\nu+\tau)^{-\sigma}]  \Psi(\nu, \tau)g(\nu, s) |Z(\tau, s)|  d\tau d\nu ds\nonumber\\
&+ \frac{1}{2} \int_{0}^{t}\int_{0}^{n}\int_{0}^{\infty} [N+\eta(\sigma)(\nu+\tau)^{-\sigma}]  \Psi(\nu, \tau) h(\tau, s) |Z(\nu, s)|  d\tau d\nu ds\nonumber\\
&+ \frac{1}{2} \int_{0}^{t}\int_{n}^{\infty}\int_{0}^{\infty} [N+\eta(\sigma)(\nu+\tau)^{-\sigma}]   \Psi(\nu, \tau)\nonumber\\
&\hspace{2cm}\times [g(\nu, s) |Z(\tau, s)|+h(\tau, s)|Z(\nu, s)|]  d\tau d\nu ds =: \sum_{i=1}^7 S_i^n(t),
\end{align}
where $S_i^n(t)$, for $i = 1, 2, 3, 4, 5, 6, 7$ are the corresponding integrands in the preceding line. We now solve each $S_i^n(t)$ individually.
Using properties of the Signum function, consider following two bounds, as
\begin{align}\label{bound1}
Q(\mu, \nu, s)  Z(\nu, s) = & \bigg[ [1+(\mu+\nu)^{-\sigma}] \mbox{sgn}(Z(\mu+\nu, s))E(\mu, \nu) -(1+\mu^{-\sigma})\mbox{sgn}(Z(\mu,s))\nonumber\\
&-(1+\nu^{-\sigma})\mbox{sgn}(Z(\nu,s)) \bigg] Z(\nu, s)\nonumber \\
 \le & \bigg[ 1+\mu^{-\sigma}+ \nu^{-\sigma}  +(1+\mu^{-\sigma}) -(1+\nu^{-\sigma}) \bigg] |Z(\nu, s)| \nonumber\\
  \le & 2 ( 1+\mu^{-\sigma} ) |Z(\nu, s)|.
\end{align}
Similarly, we have
\begin{align}\label{bound2}
Q(\mu, \nu, s)  Z(\mu, s) 
  \le  2 ( 1+\nu^{-\sigma} ) |Z(\mu, s)|.
\end{align}

Let us first estimate $S_1^n(t)$ by using \eqref{bound1} and \eqref{collisionKerneluniques}, as
\begin{align}\label{UniquenessS1}
S_1^n(t) \le &  k \int_{0}^{t}\int_{0}^{n}\int_{0}^{n}  ( 1+\mu^{-\sigma} )  (\mu+\nu)^{-\sigma} |Z(\nu, s)|   g(\mu, s)  d\nu d\mu ds\nonumber\\
\le &  k \int_{0}^{t}\int_{0}^{n}\int_{0}^{n}  ( 1+\mu^{-\sigma} )  (1+\nu^{-\sigma})  |Z(\nu, s)|   g(\mu, s)  d\nu d\mu ds\nonumber\\
 = &  k \int_{0}^{t} \Xi^{n}(s)  \int_{0}^{n}  ( 1+\mu^{-\sigma} )   g(\mu, s)  d\mu ds \le 2 k \sup_{s\in[0, t] }\| g(s)\|_{\mathcal{S}} \int_{0}^{t} \Xi^{n}(s)ds.
\end{align}
Similarly, by using \eqref{collisionKerneluniques} and \eqref{bound2}, $S_2^n(t)$ can be evaluated, as
\begin{align}\label{UniquenessS2}
S_2^n(t) 
\le 2 k \sup_{s\in[0, t] }\| h(s)\|_{\mathcal{S}} \int_{0}^{t} \Xi^{n}(s)ds.
\end{align}

Next, consider the following integral by using \eqref{collisionKerneluniques}, as
\begin{align}\label{UniquenessS30}
|S_3^n(t)| \le &  k \int_{0}^{t}\int_{0}^{\infty}\int_{0}^{\infty}  (1+\mu^{-\sigma})  (\mu+ \nu)^{-\sigma} g(\mu,s) |Z(\nu,s)| d\nu d\mu ds\nonumber\\
\le &  k \int_{0}^{t} \int_{0}^{\infty} \int_{0}^{\infty}  (1+\mu^{-\sigma})  \nu^{-\sigma} g(\mu,s) [g(\nu, s)+h(\nu, s)] d\nu d\mu ds\nonumber\\
\le & 2 kT \sup_{s\in[0, t] }\bigg[\|g(s)\|_{\mathcal{S}}+\|h(s)\|_{\mathcal{S}} \bigg]   \|g(s)\|_{\mathcal{S}}  < \infty.
\end{align}

Then, by Lemma 2.1 in \cite{Giri:2011} to \eqref{UniquenessS30}, we obtain
\begin{align}\label{UniquenessS3}
S_3^n(t) \to 0\ \ \text{as}\ n \to \infty.
\end{align}

 $S_4^n(t)$ can be evaluated by applying \eqref{collisionKerneluniques}, as
\begin{align}\label{UniquenessS4}
S_4^n(t) \leq & \frac{1}{2} k \eta(\sigma) \int_{0}^{t}\int_{0}^{n}\int_{0}^{n} [1+(\nu+\tau)^{-\sigma}]  (\nu+\tau)^{-\sigma} g(\nu, s) |Z(\tau, s)|  d\tau d\nu ds\nonumber\\
\leq & \frac{1}{2} k \eta(\sigma) \int_{0}^{t}\int_{0}^{n}\int_{0}^{n} (1+\nu^{-\sigma})  (1+\tau^{-\sigma}) g(\nu, s) |Z(\tau, s)|  d\tau d\nu ds\nonumber\\
\leq &  k \eta(\sigma) \sup_{s\in[0, t] } \|g(s)\|_{\mathcal{S}}  \int_{0}^{t} \Xi^n(s)  ds.
\end{align}

By using \eqref{collisionKerneluniques}, we estimate $S_5^n(t)$ as
\begin{align}\label{UniquenessS50}
|S_5^n(t)| \leq & \frac{1}{2} k \eta(\sigma) \int_{0}^{t}\int_{0}^{n}\int_{n}^{\infty} [1+(\nu+\tau)^{-\sigma}]  (\nu+\tau)^{-\sigma} g(\nu, s) |Z(\tau, s)|  d\tau d\nu ds\nonumber\\
\leq & \frac{1}{2} k \eta(\sigma) \int_{0}^{t}\int_{0}^{n}\int_{n}^{\infty} (1+\tau^{-\sigma})  \nu^{-\sigma} g(\nu, s) [g(\tau, s)+h(\tau, s)]  d\tau d\nu ds\nonumber\\
\leq &   k \eta(\sigma)  \sup_{s\in[0, t] }\bigg[\|g(s)\|_{\mathcal{S}} +\|h(s)\|_{\mathcal{S}}  \bigg] \sup_{s\in[0, t] } \|g(s)\|_{\mathcal{S}}  T< \infty.
\end{align}

Then, by an application of Lebesgue dominated convergence theorem to \eqref{UniquenessS50}, we obtain
\begin{align}\label{UniquenessS5}
S_5^n(t) \to 0\ \ \text{as}\ n \to \infty.
\end{align}
 Now, we evaluate $S_6^n(t)$ similar to $S_6^4(t)$ by using \eqref{collisionKerneluniques} as
\begin{align}\label{UniquenessS6}
S_6^n(t) 
\le  \frac{1}{2} k \eta(\sigma) \sup_{s\in[0, t] } \|h(s)\|_ {\mathcal{S}} \int_{0}^{t} \Xi^n(s) ds.
\end{align}
Finally, we estimate the last term $S_7^n(t)$ similar to $S_5^n(t)$ by using \eqref{collisionKerneluniques} as
\begin{align}\label{UniquenessS70}
|S_7^n(t)| 
\le   k \eta(\sigma) T \sup_{s\in[0, t] } \{ \|g(s)\|_{\mathcal{S}}+\|h(s)\|_{\mathcal{S}}\}^2 < \infty.
\end{align}

Then, by using Lebesgue dominated convergence theorem into \eqref{UniquenessS70}, we find
\begin{align}\label{UniquenessS7}
S_7^n(t) \to 0\ \ \text{as}\ n \to \infty.
\end{align}

Now, taking $n \to \infty $ to \eqref{Uniqueness5} and then inserting \eqref{UniquenessS1}, \eqref{UniquenessS2}, \eqref{UniquenessS3}, \eqref{UniquenessS4}, \eqref{UniquenessS5}, \eqref{UniquenessS6} and \eqref{UniquenessS7}, we have
\begin{align}\label{Uniqueness6}
\lim_{n \to \infty} \Xi^{n}(t)\le &  k \sup_{s\in[0, t] } \bigg[ 2  \| g(s)\|_{\mathcal{S}}  + 2  \| h(s)\|_{\mathcal{S}}  \nonumber\\
& + \eta(\sigma)  \|g(s)\|_{\mathcal{S}}
+\frac{1}{2}  \eta(\sigma) \|h(s)\|_{\mathcal{S}} \bigg] \int_{0}^{t} \lim_{n \to \infty} \Xi^n(s) ds.
\end{align}
Then applying Gronwall's inequality to \eqref{Uniqueness6}, we obtain
\begin{align*}
\lim_{n \to \infty} \Xi^{n}(t)=0\ \ \forall t.
\end{align*}
This implies $g(\mu, t)=h(\mu, t)$ almost everywhere. This completes the proof of the Theorem \ref{uniquemain theorem1}.

\end{proof}

\section*{Acknowledgments}
This work was mainly supported by Science and Engineering Research Board (SERB), Department of Science and Technology (DST), India through the project
 $YSS/2015/001306$. A part of this work is also supported by Faculty Initiation Grant (FIG: MTD/FIG/100680), Indian Institute of Technology Roorkee, India.
In addition, authors would like to thank University Grant Commission (UGC) India for providing the PhD fellowship to PKB through the Grant
 $6405/11/44$.

\bibliographystyle{plain}

\end{document}